\documentclass[final,oneside,notitlepage,10pt]{article}

\usepackage{amsthm}
\usepackage{thmtools}
\usepackage{graphicx}
\usepackage[margin=2cm,font=normalsize,labelfont=bf]{caption}
\usepackage{verbatim}
\usepackage[shortlabels]{enumitem}
\usepackage{fancyhdr}
\usepackage[]{geometry}
\usepackage{xstring}
\usepackage{needspace}
\usepackage{color}
\usepackage{amstext}
\usepackage{nameref}
\usepackage[hypertexnames=false]{hyperref}
\usepackage{mathdots}
\usepackage[normalem]{ulem}
\usepackage{chngcntr}
\usepackage[section]{placeins}
\usepackage{tikz-cd}    
  
\makeatletter
\newcounter{denseversion}
\newcounter{comments}
\newcounter{authorcounter}
\newcounter{adresscounter}

\def\title#1{\gdef\@title{#1}}
\def\@title{}
\def\subtitle#1{\gdef\@subtitle{#1}}
\def\@subtitle{}

\def\authortagsused{0}
\def\adresstag#1{\if!#1!\else$^{\;#1\;}$\fi}
\def\@authorsep#1{
  \ifnum\value{authorcounter}=#1 and \else\unskip, \fi
}

\renewcommand{\author}[2][]{
  \stepcounter{authorcounter}
  \if!#1!\else\gdef\authortagsused{1}\fi
  \ifnum\value{authorcounter}=1
    \def\@authorstringa{#2\adresstag{#1}}
    \def\@authorstringb{#2}
    \def\@authorstringc{#2\adresstag{#1}}
  \else
    \ifnum\value{authorcounter}=2
      \g@addto@macro\@authorstringa{\@authorsep{2}#2\adresstag{#1}}
      \g@addto@macro\@authorstringb{\@authorsep{2}#2}
      \g@addto@macro\@authorstringc{\\#2\adresstag{#1}}
    \else
      \g@addto@macro\@authorstringa{\@authorsep{3}#2\adresstag{#1}}
      \g@addto@macro\@authorstringb{\@authorsep{3}#2}
      \g@addto@macro\@authorstringc{\\#2\adresstag{#1}}
    \fi
  \fi}
\def\@author{\ifnum\value{denseversion}=0\@authorstringa\else\@authorstringb\fi}

\def\@adressstringa{}
\def\@adressstringb{}
\newcommand{\adress}[2][]{
  \stepcounter{adresscounter}
  \ifnum\value{adresscounter}=1
    \g@addto@macro\@adressstringa{\ifnum\authortagsused=0\def\br{\\}\else\def\br{, }\fi\adresstag{#1}#2}
    \g@addto@macro\@adressstringb{\def\br{\\}\adresstag{#1}\parbox[t]{14cm}{#2}}
  \else
    \g@addto@macro\@adressstringa{\\[\bigskipamount]\adresstag{#1}#2}
    \g@addto@macro\@adressstringb{\\[\medskipamount]\adresstag{#1}\parbox[t]{14cm}{#2}}
  \fi}

\def\preprint#1{\gdef\@preprint{#1}}
\def\@preprint{}
\def\keywords#1{\gdef\@keywords{#1}}
\def\@keywords{}
\def\msc#1{\gdef\@msc{#1}}
\def\@msc{}
\def\email#1{
   \gdef\@email{#1}
   \g@addto@macro\@authorstringc{ {\it (#1)}}}
\def\@email{}
\def\dedication#1{\gdef\@dedication{#1}}
\def\@dedication{}

\def\mybaselinestretch#1{
  \gdef\@mybaselinestretch{#1}
  \renewcommand{\baselinestretch}{\@mybaselinestretch}}

\def\myparskip#1{
  \gdef\@myparskip{#1}
  \setlength{\parskip}{\@myparskip}}

\mybaselinestretch{1.2}
\myparskip{1.5ex}

\binoppenalty=\maxdimen
\relpenalty=\maxdimen

\normalfont
\normalsize
\newlength{\@listleftmargin}
\def\setenumstandard{
  \setlist{leftmargin=\@listleftmargin,itemsep=0pt,topsep=0pt,partopsep=0pt,parsep=\@myparskip}
  \setlist[enumerate]{align=left,labelsep=*,leftmargin=\@listleftmargin,itemsep=0pt,topsep=0pt,partopsep=0pt,parsep=\@myparskip}
}

\def\denseversion{
  \setcounter{denseversion}{1}
  \newgeometry{left=3cm,right=3cm,top=3cm}
  \mybaselinestretch{1.1}
  \myparskip{0.8ex}
  \normalfont
  \def\possiblelinebreak{}
  \fancyfoot[C]{\itshape{--$\,\,$\thepage$\,\,$--}}}

\def\possiblelinebreak{\\}

\renewcommand{\emph}[1]{\def\reserved@a{it}\ifx\f@shape\reserved@a\ul{#1}\else\textit{#1}\fi}

\def\setcrefnames{}
\newcommand{\mytableofcontents}{
   \ifnum\value{denseversion}=0
     \tableofcontents
     \setcrefnames 
   \else
     \renewcommand{\baselinestretch}{1.1}
     \setlength{\parskip}{0ex}
     \normalfont
     \begingroup
     \def\addvspace##1{\vskip0.4em}
     \tableofcontents
     \setcrefnames 
     \endgroup
     \renewcommand{\baselinestretch}{\@mybaselinestretch}
     \setlength{\parskip}{\@myparskip}
     \normalfont
   \fi}

\newlength{\zeilenlaenge}
\def\putindent#1{
  \settowidth{\zeilenlaenge}{#1}
  \ifnum\zeilenlaenge>\textwidth
    #1
  \else
    \noindent #1
  \fi
}

\hypersetup{
    linktocpage = true,
    colorlinks,
    linkcolor=[RGB]{0,0,120},
    citecolor=[RGB]{0,0,120},
    urlcolor=[RGB]{0,0,120}
}

\def\pdfdaten{
  \hypersetup{
    pdftitle = {\@title},
    pdfauthor = {\@author},
    pdfkeywords = {\@keywords},    
    bookmarksopen = true,
    bookmarksopenlevel = 1
  }}  

\def\showkeywords{\begin{flushleft}\footnotesize\textbf{Keywords}: \@keywords\end{flushleft}}
\def\showmsc{\begin{flushleft}\footnotesize\textbf{MSC 2010}: \@msc\end{flushleft}}

\def\tocsection#1{\section*{#1}\addcontentsline{toc}{section}{#1}}

\def\mytitle{}
\def\zmptitle{
  \begin{tabular}{cc}
    \begin{minipage}[c]{0.4\textwidth}
      \begin{flushleft}
        \includegraphics[width=110pt]{../../tex/zmp}
      \end{flushleft}  
    \end{minipage}&
    \begin{minipage}[c]{0.55\textwidth}
      \begin{flushright}
      {\small\sf\@preprint}
      \end{flushright}
    \end{minipage}
  \end{tabular}
  \vskip 2cm}

\def\maketitle{
  \pdfdaten
  \noindent
  \mytitle
  \begin{center}
    \LARGE\@title\\
    \if!\@subtitle!\else\smallskip\LARGE\@subtitle\\\fi
    \bigskip
    \if!\@author!\else\bigskip\large\@author\\\fi
    \ifnum\value{denseversion}=0
      \if!\@adressstringa!\else\bigskip\normalsize\@adressstringa\\\fi
      \if!\@email!\else\ifnum\value{authorcounter}=1\bigskip\normalsize\textit{\@email}\\\else\fi\fi
    \else
    \fi
    \if!\@dedication!\else\bigskip\normalsize{\@dedication}\\\fi
  \end{center}
  \ifnum\value{denseversion}=0\vskip 1.5cm\else\vskip0.5cm\fi}

\def\kobib#1{
  \begin{raggedright}
  \ifnum\value{denseversion}=0\else\small\fi
  \Oldbibliography{#1/kobib}
  \bibliographystyle{#1/kobib}
  \end{raggedright}
  \ifnum\value{denseversion}=0\else
      \noindent
      \if!\@authorstringc!\else
        \ifnum\authortagsused=0\ifnum\value{authorcounter}>1\normalsize\@authorstringc\\[\medskipamount]\else\fi\else\normalsize\@authorstringc\\[\medskipamount]\fi
      \fi
      \if!\@adressstringb!\else\normalsize\@adressstringb\\{}\fi
      \ifnum\authortagsused=0
        \ifnum\value{authorcounter}=1
          \if!\@email!\else\linebreak\normalsize\textit{\@email}\\{}\fi
        \else
        \fi
      \else
      \fi
  \fi}

\let\Oldbibliography\bibliography
\def\bibliography#1{
  \begin{raggedright}
  \ifnum\value{denseversion}=0\else\small\fi
  \Oldbibliography{#1}
  \end{raggedright}
  \ifnum\value{denseversion}=0\else
      \medskip
      \noindent
      \if!\@authorstringc!\else
        \ifnum\authortagsused=0\ifnum\value{authorcounter}>1\normalsize\@authorstringc\\[\medskipamount]\else\fi\else\normalsize\@authorstringc\\[\medskipamount]\fi
      \fi
      \if!\@adressstringb!\else\normalsize\@adressstringb\\{}\fi
      \ifnum\authortagsused=0
        \ifnum\value{authorcounter}=1
          \if!\@email!\else\linebreak\normalsize\textit{\@email}\\{}\fi
        \else
        \fi
      \else
      \fi
  \fi
}

\newenvironment{commentfigure}{}
\newenvironment{sidewayscommentfigure}{\begin{minipage}}{\end{minipage}}
\newenvironment{displaycomment}{\begin{list}{}{\rightmargin=1cm\leftmargin=1cm}\item\sf\begin{small}\color{gray}}{\end{small}\end{list}}

\def\tocmark#1{}

\def\draftstamp#1{
  \def\tocmark##1{
    \ifnum\c@secnumdepth=0\section{##1}\fi
    \ifnum\c@secnumdepth=1\subsection{##1}\fi
    \ifnum\c@secnumdepth=2\subsubsection{##1}\fi
    \ifnum\c@secnumdepth=3\subsubsection{##1}\fi
  }
  \ifnum\value{comments}=0
    \gdef\@draft{DRAFT - Edited on \today\ by #1 - Comments are not displayed}
  \else
    \gdef\@draft{DRAFT - Edited on \today\ by #1 - Comments are displayed}
  \fi
  \fancyhead[C]{\footnotesize\tt\textcolor{red}{\@draft}}}

\def\skript{
  \renewenvironment{displaycomment}{}{}
  \ifnum\value{comments}=0
    \renewenvironment{example*}{\comment}{\endcomment}
    \renewenvironment{remark*}{\comment}{\endcomment}
  \else\fi
  \parindent=0mm        
}

\mybaselinestretch{}
\setenumstandard

\makeatother

\input{kobib}

\usepackage{amssymb}
\usepackage{amsmath}
\usepackage{amstext}
\usepackage{mathrsfs}
\usepackage{euscript}
\usepackage[all,cmtip]{xy}
\usepackage{xstring}
\usepackage{slashed}
\usepackage{tensor}
\usepackage{mathtools}

\def\ul{\underline}
\entrymodifiers={+!!<0pt,\fontdimen22\textfont2>}

\def\N {\mathbb{N}}
\def\Z {\mathbb{Z}}

\def\R {\mathbb{R}}
\def\C {\mathbb{C}}

\def\hc#1{\mathrm{h}_{#1}}
\def\h {\mathrm{H}}

\def\subset{\subseteq}

\makeatletter
\renewenvironment{proof}[1][\nameProof]
  {\par\pushQED{\qed}%
   \normalfont \topsep6\p@\@plus6\p@\relax
   \trivlist
   \item[\hskip\labelsep
         \itshape
         #1\@addpunct{.}]
  \leavevmode}
  {\popQED\endtrivlist\@endpefalse}
\makeatother

\def\notebox#1#2{\begin{minipage}[b]{#1}\sloppy\renewcommand{\baselinestretch}{0.8}\footnotesize \begin{center}#2\end{center}\end{minipage}}

\newcommand{\arr}[1][r]{\ar@<0.7ex>[#1]\ar@<-0.7ex>[#1]}
\newcommand{\arrr}[1][r]{\ar@<1.4ex>[#1]\ar[#1]\ar@<-1.4ex>[#1]}

\newlength{\myeqt} 
\newlength{\myeqs} 
\newlength{\myeqm} 
\newlength{\myeqn} 
\newcommand\symtext[3][\myeqn]{
  \settowidth{\myeqt}{#2}
  \settowidth{\myeqs}{$#3$}
  \addtolength{\myeqs}{\the\myeqm}
  \ifdim\myeqt>\myeqs
    \stackrel{\hspace{-#1}\notebox{#1}{\medskip #2 \\ $\downarrow$\smallskip}\hspace{-#1}}{#3}
  \else
    \stackrel{\text{#2}}{#3}
  \fi}

\def\brackets#1{\IfStrEq{#1}{-}{}{(#1)}}
\def\subindex#1{\IfStrEq{#1}{-}{}{_{#1}}}



\newlength{\myl}

\def\ddt#1#2#3{\left.\frac{\mathrm{d}^{\IfStrEq{#1}{1}{}{#1}}}{\mathrm{d}#2}\IfStrEq{#2}{#3}{\right.}{\right|_{#3}}}

\def\lw#1#2{{}^{#1\!}#2}

\newlength{\widthtmp}
\def\length#1{\settowidth{\widthtmp}{#1}\the\widthtmp}

\definecolor{olivegreen}{rgb}{.33,.55,.18}

\newcommand{\ie}{i.e., }
\newcommand{\eg}{e.g., }

\newcommand{\mc}[1]{\mathcal{#1}}

\def\pss#1{\prescript{}{#1}}

\newcommand{\GL}{\operatorname{GL}}

\newcommand{\Cl}{\operatorname{Cl}}
\newcommand{\CCl}{\C\!\operatorname{l}}

\newcommand{\opp}{\operatorname{op}}
\newcommand{\pr}{\operatorname{pr}}
\newcommand{\lact}{\triangleright}
\newcommand{\ract}{\triangleleft}

\newcommand{\id}{\operatorname{id}}
\newcommand{\unit}{\mathbf{1}}

\newcommand{\Homcat}{\mathscr{H}\mathrm{om}}
\newcommand{\Endcat}{\mathscr{E}\!\mathrm{nd}}
\newcommand{\Hom}{\mathrm{Hom}}
\newcommand{\End}{\mathrm{End}}
\newcommand{\Aut}{\mathrm{Aut}}
\newcommand{\Out}{\mathrm{Out}}
\newcommand{\AUT}{\mathscr{A}\mathrm{ut}}

\newcommand{\Vect}{\mathscr{V}\mathrm{ect}}

\newcommand{\sVect}{\mathrm{s}\Vect}

\newcommand{\Bimod}{\mathscr{B}\mathrm{imod}}

\newcommand{\BimodBdl}{\mathscr{B}\mathrm{im}\mathscr{B}\mathrm{dl}}
\newcommand{\sBimodBdl}{\mathrm{s}\BimodBdl}
\newcommand{\sBimodBdlimp}{\mathrm{s}\BimodBdl^{\mathrm{imp}}}

\newcommand{\Alg}{\Incl\mathrm{lg}}       
\newcommand{\Algbi}{\Alg^{\mathrm{bi}}}   
\newcommand{\Algbiimp}{\Alg^{\mathrm{imp}}}   
    
\newcommand{\Algbigrpd}[1]{\mathrm{Grpd}(\Alg^{\mathrm{bi}}_{#1})}   
\newcommand{\sAlg}{\mathrm{s}\Alg}    
\newcommand{\sAlggrpd}[1]{\mathrm{Grpd}(\mathrm{s}\Alg_{#1})}    
\newcommand{\sAlgbi}{\sAlg^{\mathrm{bi}}}   
\newcommand{\sAlgbigrpd}[1]{\mathrm{Grpd}(\sAlg^{\mathrm{bi}}_{#1})}   
\newcommand{\sAlgbiimp}{\sAlg^{\mathrm{imp}}}   
\newcommand{\csAlg}{\mathrm{cs}\Incl\mathrm{lg}}       
\newcommand{\cssAlg}{\mathrm{css}\Alg}       
\newcommand{\cssAlgbi}{\mathrm{css}\Algbi}       
\newcommand{\ssAlg}{\mathrm{ss}\Alg}       
\newcommand{\ssAlgbi}{\mathrm{ss}\Algbi}       
\newcommand{\sssAlg}{\mathrm{ss\text{-}s}\Alg}       
\newcommand{\sssAlgbi}{\mathrm{ss\text{-}s}\Algbi}       

\newcommand{\AlgBdl}{\Alg\mathscr{B}\mathrm{dl}}       
\newcommand{\sssAlgBdl}{\mathrm{ss\text{-}s}\AlgBdl}       
\newcommand{\ssAlgBdl}{\mathrm{ss}\AlgBdl}       
\newcommand{\AlgBdlbi}{\Alg\mathscr{B}\mathrm{dl}^{\mathrm{bi}}}
\newcommand{\sAlgBdlbi}{\sAlg\mathscr{B}\mathrm{dl}^{\mathrm{bi}}}
\newcommand{\sAlgBdlbigrpd}[2]{\mathrm{Grpd}(\sAlg\mathscr{B}\mathrm{dl}^{\mathrm{bi}}_{#1}\brackets{#2})} 
\newcommand{\AlgBdlbigrpd}[2]{\mathrm{Grpd}(\Alg\mathscr{B}\mathrm{dl}^{\mathrm{bi}}_{#1}\brackets{#2})} \newcommand{\sssAlgBdlbi}{\sssAlg\mathscr{B}\mathrm{dl}^{\mathrm{bi}}}
\newcommand{\ssAlgBdlbi}{\ssAlg\mathscr{B}\mathrm{dl}^{\mathrm{bi}}}      
\newcommand{\sAlgBdl}{\sAlg\mathscr{B}\mathrm{dl}}       
\newcommand{\sAlgBdlgrpd}[2]{\mathrm{Grpd}(\sAlg\mathscr{B}\mathrm{dl}_{#1}\brackets{#2})} 
\newcommand{\AlgBdlgrpd}[2]{\mathrm{Grpd}(\Alg\mathscr{B}\mathrm{dl}_{#1}\brackets{#2})}

\newcommand{\cssAlgBdlbi}{\cssAlg\mathscr{B}\mathrm{dl}^{\mathrm{bi}}} 
\newcommand{\csAlgBdlbi}{\csAlg\mathscr{B}\mathrm{dl}^{\mathrm{bi}}}       

\newcommand{\Impcat}{\mathscr{I}\mathrm{mp}}

\newcommand{\Mfd}{\mathscr{M}\mathrm{fd}}

\newcommand{\Des}{\mathscr{D}\mathrm{es}}

\newcommand{\Incl}{\mathscr{A}}

\usepackage[latin1]{inputenc}
\usepackage[english]{babel}
\usepackage[english]{cleveref}

\def\refname{References}

\def\quot#1{``#1''}

\def\quand{\quad\text{ and }\quad}
\def\quomma{\quad\text{, }\quad}

\def\nameTheorem{Theorem}
\def\nameTheorems{Theorems}
\def\nameDefinition{Definition}
\def\nameDefinitions{Definitions}
\def\nameCorollary{Corollary}
\def\nameCorollaries{Corollaries}
\def\nameProposition{Proposition}
\def\namePropositions{Propositions}
\def\nameConstruction{Construction}
\def\nameConstructions{Constructions}
\def\nameLemma{Lemma}
\def\nameLemmas{Lemmas}
\def\nameRemark{Remark}
\def\nameRemarks{Remarks}
\def\nameProblem{Problem}
\def\nameProblems{Problems}
\def\nameExample{Example}
\def\nameExamples{Examples}
\def\nameExercise{Exercise}
\def\nameExercises{Exercises}

\def\nameProof{Proof}
\def\nameEquation{Eq.}
\def\nameEquations{Eqs.}
\def\nameSection{Section}
\def\nameSections{Sections}
\def\nameAppendix{Appendix}
\def\nameAppendixs{Appendices}
\def\nameFigure{Figure}
\def\nameFigures{Figures}

\input{theorems}

\def\mathscr#1{\EuScript{#1}}

\usepackage{tikz-cd}
\usepackage{todonotes}

\title{The insidious bicategory of algebra bundles}

\author[a]{Peter Kristel}
\email{peter.kristel@umanitoba.ca}

\adress[a]{
  University of Manitoba\\
  Department of Mathematics\\
  Winnipeg, MB R3T 2N2, Canada}

\author[b]{Matthias Ludewig}
\email{matthias.ludewig@mathematik.uni-regensburg.de}

\adress[b]{
Universit\"at Regensburg\\
  Fakult\"at f\"ur Mathematik\\
93053 Regensburg, Germany}

\author[c]{Konrad Waldorf}
\email{konrad.waldorf@uni-greifswald.de}

\adress[c]{
  Universit\"at Greifswald\\
  Institut f\"ur Mathematik und Informatik\\
  D-17487 Greifswald}

\msc{}
\dedication{}
\pdfdaten

\denseversion
\allowdisplaybreaks
\usepackage{amstext}
\usepackage{amssymb}

\begin{document}

\maketitle

\begin{abstract}
\noindent
In this paper we construct a bicategory of (super) algebra bundles over a smooth manifold, where the  1-morphisms are  bundles of bimodules. The main point is that naive definitions of bimodule bundles will not lead to a well-defined composition law in such a bicategory, at least not if non-invertible bimodules and non-semisimple algebras are desired. This  problem  has not been addressed so far in the literature. We develop a complete solution, and also address symmetric monoidal structures as well as the corresponding questions of dualizability.      
\end{abstract}

\mytableofcontents

\setsecnumdepth{1}

\section{Introduction}

Algebras over a field $k$ usually form a category $\Alg_k$ whose morphisms are algebra homomorphisms, but they also form a bicategory $\Algbi_k$, whose  1-morphisms are bimodules and whose 2-morphisms are bimodule intertwiners. 
The composition is the relative tensor product of two bimodules over an algebra. 
In the context of smooth functorial field theories, the need of a \emph{bundle version} of this bicategory arose, with objects algebra bundles and 1-morphisms bimodule bundles. 
This bicategory is usually considered to be straightforward to define and not worth the paper needed to spell out the details.

The present article is devoted to the fact that this popular belief is misguided. 
While it is correct that algebra bundles and bimodule bundles can be defined in a straightforward way, it is not correct that such bimodule bundles admit a relative tensor product over an algebra bundle; hence, there is no well-defined composition law when the bicategorical structure is set up naively. 
We have not been able to find any notice of this problem anywhere in the literature, let alone a solution.

The problem lies in the construction of the relative tensor product as the quotient of the ordinary tensor product by a subspace generated from the algebra actions. For bundles, one has to take this quotient fibrewise. Now, when the algebra actions vary while moving through base space, the dimension of above-mentioned subspace is not necessarily locally constant; thus, the resulting object will not even have an underlying locally trivial vector bundle.

One might now try to control the algebra actions so that the dimension does not jump. To be more precise, suppose $\mathcal{A}$ and $\mathcal{B}$ are algebra bundles over a smooth manifold $X$, and $\mathcal{M}$ is a vector bundle over $X$ on which $\mathcal{A}$ and $\mathcal{B}$ act in terms of commuting algebra bundle homomorphisms
\begin{equation*}
\mathcal{A} \to \mathrm{End}(\mathcal{M})
\quand
\mathcal{B}^{\opp} \to \mathrm{End}(\mathcal{M})\text{.} 
\end{equation*}  
A reasonable condition to control these actions is to require that $\mathcal{A}$, $\mathcal{B}$, and $\mathcal{M}$ have typical fibres: algebras $A$ and $B$, and an $A$-$B$-bimodule $M$ such that, around each point of $X$ there are local trivializations
\begin{equation*}
\mathcal{A}|_U \cong U \times A
\quomma
\mathcal{B}|_U \cong U \times B
\quand
\mathcal{M}|_U \cong U \times M
\end{equation*} 
that are fibrewise algebra isomorphisms and bimodule intertwiners along algebra homomorphisms, respectively. This is precisely our  \cref{def:bimodulebundle} of an $\mathcal{A}$-$\mathcal{B}$-bimodule bundle.
Surprisingly, imposing the existence of typical fibres does \emph{not} solve the problem with the relative tensor product. 
An example where the fiberwise relative tensor product of bimodules in the above sense has varying fibre dimensions, hence admits no vector bundle structure, is given in \cref{ExNoTypicalFiber}.

In this article, we propose a further, new condition to be put on bimodules, namely we require them to be \emph{implementing}.
This requires for an $A$-$B$-bimodule $M$ that certain automorphisms of $A$ and $B$ are implemented as automorphisms of $M$. 
Here, an algebra automorphism $\varphi\in \Aut(A)$ is called implemented if there exists  a linear automorphism $u:M \to M$ such that
\begin{equation*}
\varphi(a) \lact m = u^{-1}( a \lact u(m)) 
\end{equation*} 
for all $a\in A$ and $m\in M$. 
We prove then that the relative tensor product of implementing bimodules is again implementing (\cref{PropositionTensorProdImplementing}). 
Thus, we obtain a new bicategory $\Algbiimp_k$ of algebras, implementing bimodules, and intertwiners. 
It is a proper sub-bicategory of the usual bicategory $\Algbi_k$, and the main objective of this article is to show that (only)  this sub-category admits a bundle version.

Our main result, \cref{prop:tensorproductimplementing}, is that the relative tensor product of  implementing bimodule bundles is again an  implementing bimodule bundle. As a consequence, we obtain a well-defined bicategory $\AlgBdlbi_k(X)$ of algebra bundles, implementing bimodule bundles, and intertwining bundle morphisms (\cref{DefinitionPrestwoVectBdl}). We believe that this is the first time that such a bicategory has been properly constructed.

Let us remark at this point that our new requirement for bimodules to be implementing is in fact not very restrictive. We offer two results supporting this:
\begin{enumerate}[(1)]

\item 
Every \emph{invertible} bimodule is implementing (\cref{PropositionInvertibleImplementing}). In particular, we do not change the notion of isomorphism (a.k.a. Morita equivalence) in any of our bicategories.

\item
        Every bimodule over \emph{semisimple} algebras is implementing. Note that semisimple is the same as separable, when working over a field. More generally, we show that every bimodule over algebras with vanishing first Hochschild cohomology is implementing (\cref{PropHochschildImplementing}).

\end{enumerate} 

We also have to point out a disadvantageous consequence of our solution. 
An important feature of the bicategory $\Algbi_k$ of algebras and  bimodules is that it is symmetric monoidal. 
However, the monoidal structure does not restrict to our new sub-bicategory $\Algbiimp_k$ of algebras and implementing bimodules, because the exterior tensor product of implementing bimodules does not need to be implementing (\cref{ExampleImplementing3}). 
This means that our bicategory $\AlgBdlbi_k(X)$ of algebra bundles is \emph{not} monoidal. 
Even more, it means that there is no  monoidal bicategory of algebra bundles at all. However, based on the above-mentioned results, we offer two possible solutions (\cref{prop:symmetricmonoidalstructures}): 
\begin{enumerate}[(a)]

\item 
The underlying sub-bigroupoid $\AlgBdlbigrpd kX$, obtained  by discarding all non-invertible bimodule bundles, is symmetric monoidal.  

\item
The full sub-bicategory $\ssAlgBdlbi_k(X)$ over all algebra bundles with semisimple fibres is symmetric monoidal.

\end{enumerate}
We also investigate the invertible objects in both of these symmetric monoidal bicategories, and find (\cref{lemma:invertibleobjectsinbundles}) that they are given in both cases by algebra bundles with \emph{central simple} fibres. These are classical objects that have been studied and classified by Donovan-Karoubi  \cite{DK70} in the context of topological K-theory.

The quest for symmetric monoidal structures is particularly interesting in the context of \emph{2-vector spaces}. A 2-vector space over a field $k$ is an object in  any symmetric monoidal bicategory $\mathscr{C}$ with $\Endcat_{\mathscr{C}}(\unit,\unit)\cong \Vect_k$, where $\unit$ is the tensor unit.  For the bicategory $\Algbi_k$ of algebras and bimodules, this condition is satisfied; hence, $\Algbi_k$ is a bicategory of 2-vector spaces, and every algebra is a 2-vector space. 
One motivation for this work was the question if the bicategory $\AlgBdlbi_k (X)$  defined here models \emph{2-vector bundles} in analogous way. This idea goes back to contributions of Schreiber at the n-category Caf\'{e} \cite{Schreiber2006,Schreiber2007}, 
see also \cite[\S A]{Schreiber2009} and \cite[\S 4.4]{schreiber2}. 
An important feature of 2-vector bundles (rather, of any bicategorical geometric structure) should be that they form a 2-stack, i.e., they satisfy a local-to-global principle, descent to quotients, etc. 
We thus study our bicategory of algebra bundles in dependence of the manifold $X$, i.e., as a presheaf
\begin{equation*}
X \mapsto \AlgBdlbi_k(X)
\end{equation*} 
 of bicategories on the category of smooth manifolds. We prove (\cref{lem:pre2stack})
that it is a \emph{pre-2-stack}, but we notice (\cref{ex:nota2stack}) that -- in general -- it is not a 2-stack. A further treatment is deferred to our article \cite{Kristel2020}, in which we apply the plus construction of Nikolaus-Schweigert \cite{nikolaus2} to the pre-2-stack $\AlgBdlbi_k$ constructed here. The result is a well-defined 2-stack of 2-vector bundles.

Finally, we remark that we work throughout with \emph{super} algebras, considering the ungraded situation as a special case (everything is concentrated in even degrees). Most of the problems we discuss here, in particular those concerning the relative tensor product of bimodule bundles, appear independently of whether or not algebras are graded. 

The organization of this article is as follows.
In \cref{Section2VectorSpaces} we review the usual bicategory of (super) algebras $\sAlgbi_k$, including its symmetric monoidal structure and a discussion of its dualizable, fully dualizable, and invertible objects. We also compute the automorphism 2-groups of objects of $\sAlgbi_k$, and show that they can be realized as crossed modules of Lie groups (\cref{prop:aut}).  In \cref{SectionImplementingModules} we develop our new bicategory $\sAlgbiimp_k$ centered around the notion of implementing bimodules. In \cref{sec:algebrabundles} we introduce the bundle version of the bicategory $\sAlgbiimp_k$, the bicategory of super algebra bundles $\sAlgBdlbi_k(X)$. Our main result, the well-definedness of the composition in this bicategory (\cref{prop:tensorproductimplementing}) is in \cref{sec:implementingbimodulebundles}.

\paragraph{Acknowledgements.}
We would like to thank Severin Bunk, Christoph Schweigert, and Danny Stevenson for helpful discussions. We thank Yves de Cornulier for describing some beautiful (counter-) examples on Mathoverflow \cite{Cornulier,Cornuliera}, and Jeremy Rickard for explaining to us a result about Picard-surjectivity \cite{Rickard}.
PK gratefully acknowledges support from the Pacific Institute for the Mathematical Sciences in the form of a postdoctoral fellowship. 
ML gratefully acknowledges support from SFB 1085 ``Higher invariants''.


\setsecnumdepth{2}

\section{The bicategory of algebras} \label{Section2VectorSpaces}

\label{sec:bicatalgebras}

In this section, we review the well-known symmetric monoidal framed bicategory of algebras, bimodules, and intertwiners. 
Throughout this paper, all algebras and bimodules will be finite-dimensional over the field $k=\R$ or $k=\C$.

\subsection{Algebras and bimodules} \label{sec:AlgebrasAndBimodules}

We consider only $\Z_2$-graded, unital, associative algebras $A$ (in short: super algebras).
If $V$ is a $\Z_{2}$-graded vector space, then $V^{0}$ is the even part of $V$ and $V^{1}$ is the odd part of $V$.
Moreover, if $v \in V$, then the notation $|v|$ implies that $v$ is assumed to be homogeneous, and it indicates the degree of $v$, i.e. $|v| = 0$ if $v \in V^{0}$ and $|v| = 1$ if $v \in V^{1}$.
We will now set up a bicategory $\sAlgbi_k$ whose objects are super algebras.   
 
 If $A$ and $B$ are super algebras, then the 1-morphisms $M:A \to B$  are $\Z_2$-graded, finitely generated $B$-$A$-bimodules $M$ (we will just say \emph{bimodules}).  
The 2-morphisms are bimodule intertwiners, and are always required to be parity-preserving.

If $M:B \to A$ and $N:C \to B$ are 1-morphisms (i.e.,  $M$ is a $A$-$B$-bimodule and $N$ is a $B$-$C$-bimodule), then the composition in $\sAlgbi_k$ is the relative tensor product $M \otimes_B N$ over $B$, which gives an $A$-$C$-bimodule and hence a 1-morphism $C \to A$.
The reason that 1-morphisms from $B$ to $A$ are taken to be $A$-$B$-bimodules (as opposed to $B$-$A$-bimodules), is that we then have $M\circ N = M \otimes_B N$ for the composition of 1-morphisms in our bicategory; consequently we do not have to distinguish between the symbols $\circ$ and $\otimes$. The \emph{identity bimodule} of a super algebra $A$ is  $A$, considered as an $A$-$A$-bimodule in the obvious way.

It is well-known and straightforward
to verify that these definitions yield a bicategory, the bicategory of super algebras $\sAlgbi_k$. 
Note that an $A$-$B$-bimodule $M$ is \emph{invertible} if and only if there exists a $B$-$A$-bimodule $N$ such that $M \otimes_B N \cong A$ and $N \otimes_A M \cong B$.
Thus, two super algebras $A$, $B$ are isomorphic in the bicategory $\sAlgbi_k$ if and only if there exists an invertible $A$-$B$-bimodule $M$; this relation is usually called \emph{Morita equivalence}.

Another bicategory $\Algbi_k$ of (non-super) algebras,  bimodules and intertwiners is defined as the sub-bicategory consisting of purely even super algebras and purely even bimodules. This way, we consider the ungraded situation as being included into our discussion.

\begin{remark} \label{RemarkFinDimBanach}
One reason to restrict to \emph{finite-dimensional} algebras and bimodules is that it is straightforward to talk about smooth bundles with these as fibres.  
We will later also use the fact that finite-dimensional algebras always admit submultiplicative norms (norms satisfying $\|ab\| \leq \|a\|\|b\|$), in other words, they admit the structure of a Banach algebra.
Such a norm can be obtained for example by choosing a scalar product on $A$ and then embedding $A$ into the Banach algebra $\End(A)$ of vector space endomorphisms by sending $a \mapsto \ell_a$, where $\ell_a : b \mapsto ab$ is the left multiplication operator.
We remark that since $\End(A)$ is a $C^*$-algebra, such a norm also satisfies the $C^*$-identity, but $A$ is not necessarily $*$-closed when viewed as subset of  $\End(A)$.
\end{remark}

We describe two well-known constructions with bimodules that will be relevant later. The first is particular to the graded setting. Suppose $M$ is an  $A$-$B$-bimodule. Then, the $A$-$B$-bimodule $\Pi M$ is defined to be the vector space $M$, equipped with the opposite grading, \ie $(\Pi M)_{0} := M_1$ and $(\Pi M)_1 := M_0$, the same left action, and the right action given by $m \ract_{\Pi M} b := (-1)^{|b|}m\ract_M b$. We record the following lemma, which is  straightforward to prove.

\begin{lemma}
\label{lem:Pi}
Suppose $A,B,C$ are super algebras, $M$ is an $A$-$B$-bimodule, and $N$ is a $B$-$C$-bimodule.
\begin{enumerate}[(a)]

\item 
The identity map $M \otimes_B N \to M \otimes_B N$ induces an invertible even intertwiner 
\begin{equation*}
M \otimes_B \Pi N \cong \Pi(M \otimes_B N)\text{.}
\end{equation*}

\item \label{LemmaPiB}
The map $M\otimes_B N \to M\otimes_B N:x\otimes y \mapsto (-1)^{|y|}x \otimes y$ induces an invertible even intertwiner 
\begin{equation*}
\Pi M \otimes_B N \cong \Pi(M \otimes_B N)\text{.}
\end{equation*}
\end{enumerate}
\end{lemma}

The second construction is the usual way of twisting bimodules by super algebra homomorphisms (\ie even algebra homomorphisms). 
Given a super algebra homomorphism $\phi:A' \to A$ and an $A$-$B$-bimodule $M$, there is an $A'$-$B$ bimodule $\pss{\phi} M$ with underlying vector space $M$, the right $B$-action  as before and the left $A'$-action induced along $\phi$. 
Further, if $\psi: B' \rightarrow B$ is another super algebra homomorphism, then we obtain a $A$-$B'$-bimodule $M_\psi$ in a similar way. 
Both twistings can be performed simultaneously, resulting in an $A'$-$B'$-bimodule $\pss{\phi}M_{\psi}$.
If $N$ is another $A'$-$B'$-bimodule, then an intertwiner $u: N \to \pss{\phi}M_{\psi}$ is the same as an {intertwiner along $\phi$ and $\psi$}, \ie an even linear map satisfying 
\begin{equation} \label{IntertwiningCondition}
u(a \lact n \ract b) = \phi(a)\lact u(n) \ract \psi(b), \qquad n\in N, ~~ a\in A', ~~b\in B'.
\end{equation} 
Here, we note the following rules, which are straightforward to check.  

\begin{lemma} 
Let $A$, $B$, $C$, $D$ be super algebras.
\label{lem:framing}
\begin{enumerate}[(a)] 

\item 
\label{lem:framing:a}
Let $\phi:A \to B$ and $\psi:B \to C$ be super algebra homomorphisms and let $M$ be a $D$-$C$-bimodule. 
Then the map
\begin{equation*}
c_{\phi,\psi}: M_{\psi} \otimes_B  B_{\phi} \to M_{\psi \circ \phi}: m \otimes b \mapsto m \ract \psi(b)
\end{equation*}
is well defined and an invertible intertwiner of $D$-$A$-bimodules. 

\item 
\label{lem:framing:b}
If $\phi:A \to B$,  $\psi:B \to C$ and $\tau: C \to D$ are super algebra homomorphisms, there is a commutative diagram
\begin{equation} \label{CompositorCoherence}
\xymatrix@C=5em@R=3em{D_{\tau} \otimes_C C_{\psi} \otimes_B B_{\phi} \ar[d]_-{c_{\psi,\tau} \otimes \id_B} \ar[r]^-{\id_D \otimes c_{\phi,\psi}} &  D_{\tau} \otimes_C C_{\psi \circ \phi} \ar[d]^{c_{\psi\circ\phi,\tau}} \\ D_{\tau \circ \psi} \otimes_B B_{\phi} \ar[r]_{c_{\phi,\tau\circ\psi}} & D_{\tau\circ \psi \circ\phi}\text{.}} 
\end{equation}

\item \label{LammaIdentityBimoduleInvertible}
Suppose $\phi:A \to B$ is a super algebra homomorphism. Then, the bimodule $B_{\phi}$ is invertible if and only if $\phi$ is an isomorphism. In that case, $A_{\phi^{-1}}$ is an inverse bimodule. 

\item
\label{lem:framing:d}
Suppose $\phi,\psi:A \to B$ are super algebra homomorphisms. Then, the map
\begin{equation*}
\{b\in B_0 \mid \forall a\in A: \phi(a)b=b\psi(a) \} \to \mathrm{Hom}_{B\text{-}A}(B_{\phi},B_{\psi}):b \mapsto r_b 
\end{equation*}
with $r_b(x) := xb$ is a bijection.
It restricts to a further bijection
\begin{equation*}
\{b\in B_0^{\times} \mid \phi=i(b)\circ \psi\} \cong \mathrm{Iso}_{B\text{-}A}(B_{\phi},B_{\psi})\text{,}
\end{equation*}
where $B_0^{\times}$ denotes the group of even units in $B$, and $i(b)$ denotes the inner automorphism associated to $b$, given by $i(b)(c) = bcb^{-1}$.  
In particular, if $\phi\in \mathrm{Aut}(A)$ is a super algebra automorphism, then there exists an isomorphism $A_{\phi} \cong A$ of $A$-$A$-bimodules if and only if $\phi$ is inner. 

\end{enumerate}
\end{lemma}

As a particular case of \cref{lem:framing:a}, we obtain from $\phi,\psi\in \Aut(A)$ an isomorphism $A_\phi \otimes_A A_\psi \cong A_{\phi \circ \psi}$.
Since by \cref{LammaIdentityBimoduleInvertible}, each bimodule $A_\phi$ is invertible, there is a canonical homomorphism
\begin{equation} \label{AutPicMap}
  \Aut(A) \to \mathrm{Pic}(A),
\end{equation}
where the Picard group $\mathrm{Pic}(A)$ is the group of isomorphism classes of invertible $A$-$A$-bimodules, with group multiplication given by the relative tensor product over $A$.

\begin{definition}[Picard-surjectivity] \label{DefinitionPicardSurjective}
A super algebra is called \emph{Picard-surjective}, if every invertible $A$-$A$-bimodule $M$ is isomorphic to one of the form $A_\phi$ for some $\phi \in \Aut(A)$, \ie when the map \eqref{AutPicMap} is surjective.
\end{definition}

For example, the purely even super algebra $\CCl_0=\C$ is not Picard-surjective (consider the $\C$-$\C$-bimodule $\Pi \C$), but $\CCl_1$ is (see below).
An example for an algebra that is not Picard surjective in the ungraded setting is the algebra $A=k\oplus M_2(k)$ \cite[Example 1.7]{Bolla1984}. 
Picard-surjectivity is not invariant under Morita equivalences; in fact, we show in \cref{sec:pssa}, \cref{prop:picsur}, that every Morita equivalence class of super algebras contains a Picard-surjective super algebra.

\cref{lem:framing:a} can be rephrased in the following way.
Denoting by $\sAlg_k$ the \emph{ordinary} category of super algebras and super algebra homomorphisms, we obtain a functor
\begin{equation}
\label{eq:framing2vect}
\sAlg_k \to \sAlgbi_k \text{.}
\end{equation}
Here, by functor we actually mean a 2-functor, where $\sAlg_k$ is considered as a bicategory with only identity morphisms; under this functor, composition is only preserved up to coherent compositors (sometimes such a functor is also called \emph{weak functor} or \emph{pseudo-functor}).
The functor in \cref{eq:framing2vect} is the identity on objects, and sends a super algebra homomorphism $\phi:A \to B$ to the $B$-$A$-bimodule $B_{\phi}$, which is a 1-morphism $A \to B$ in $\sAlgbi_k$. 
Note that the identity morphism $\id_A: A \to A$ is sent  to the identity bimodule $A$. 
The intertwiners $c_{\phi,\psi}$ of \cref{lem:framing:a} provide the compositors for this functor, and the diagram \eqref{CompositorCoherence} is the required coherence condition.

This situation fits into the notion of a \emph{framing}, which we recall below.
Framings of bicategories are useful because they are a tool to reduce  -- when possible -- to a one-categorical context. 

\begin{definition}[Framing]
Let $\mathscr{B}$ be a bicategory. 
A category $\mathscr{C}$ together with a functor $F: \mathscr{C} \to \mathscr{B}$ is called a \emph{framing} of $\mathscr{B}$, if $\mathscr{C}$ has the same class of objects as $\mathscr{B}$, $F$ is the identity on the level of objects, and the image of every morphism $\phi$ of $\mathscr{C}$ under $F$ has a right adjoint in $\mathscr{B}$.
A bicategory together with a framing is called a \emph{framed bicategory}.
\end{definition}

\begin{lemma}
\label{lem:framingAlg}
The functor $\sAlg_k \to \sAlgbi_k$ of \cref{eq:framing2vect} is a  framing, turning $\sAlgbi_k$ into a framed bicategory.
\end{lemma}

\begin{proof}
We have to show that for a  super algebra homomorphism $\phi: B \to A$, the $B$-$A$-bimodule $B_\phi$ admits a right adjoint.
We claim that such a right adjoint is the $A$-$B$-bimodule $_\phi B$.
Indeed, the evaluation $\varepsilon : B_\phi \otimes_A {}_\phi B \Rightarrow B$ is defined by $\varepsilon(b \otimes c) = bc$, while the coevaluation $\eta:  A \Rightarrow {}_\phi B \otimes_B B_\phi$ is given by $\eta(a) \mapsto \phi(a) \otimes 1$.
A straightforward calculation shows that the composition $(\id \otimes \varepsilon)(\eta \otimes \id)$ and $(\varepsilon \otimes \id)(\id \otimes \eta)$ are identities, as required.
\end{proof}

\begin{remark}
\label{re:doublecategory}
What we called \emph{framing} above is sometimes called a \emph{proarrow equipment}, or similar. The usual definition of a proarrow equipment is slightly more general: 
there, $\mathscr{C}$ is a strict bicategory and the framing $F$ is supposed to be locally fully faithful.
 If $\mathscr{C}$ is a category, as in our case, it extends uniquely to a strict bicategory such that $F$ is locally fully faithful.  
We remark further that framed bicategories $F:\mathscr{C} \to \mathscr{B}$ are equivalent to  \emph{double categories with companions and conjoints}. In our case, the framed bicategory $\sAlgbi_k$ corresponds to a double category whose category of  objects is $\sAlg_k$, and whose category of morphisms  consists of bimodules and intertwiners along algebra homomorphisms, then pictured as 2-cells
\begin{equation} 
\label{IntertwinerAlongHom}
\xymatrix@C=0.5cm@R=1.2cm{
  A \ar[rr]^{M} \ar[d]_{\phi_A} &  & B \ar[d]^{\phi_B}\\
  A^\prime \ar[rr]_{M^\prime} & \, & B^\prime.}
\end{equation}
The framing functor  provides companions to all vertical morphisms, and their adjoints provide conjoints. 
We refer to \cite{shulman1} for a comprehensive discussion of the relation between framed bicategories and double categories. We decided to use framed bicategories instead of double categories because then we are free to add or drop the framing at any time. 
\end{remark}

\subsection{Symmetric monoidal structure}

An important feature of the bicategory $\sAlgbi_k$ is that it is symmetric monoidal, in the sense of Schommer-Pries \cite[Definition 2.3]{pries1}. The monoidal structure is the \emph{graded} tensor product of super algebras and the exterior graded tensor product of bimodules over $k$. 
%
%
The tensor unit is the field $k$, considered as a trivially graded super algebra, $\unit := k$. 

The framing $\sAlg_k \to \sAlgbi_k$ is a framing of symmetric monoidal bicategories, which just means that it is a symmetric monoidal functor. This basically means that super algebra homomorphisms $\phi:A \to B$ and $\phi':A' \to B'$ induce invertible intertwiners
\begin{equation*}
B_{\phi} \otimes_k B'_{\phi'} \cong (B \otimes_k B')_{\phi \otimes \phi'} \end{equation*}
of $(A \otimes A')$-$(B \otimes B')$-bimodules, coherent with compositors and associators.

\begin{remark}
Note that 
\begin{equation*}
\Endcat_{\sAlgbi_k}(\unit) \cong \sVect_k
\end{equation*}
as symmetric monoidal categories, where $s\Vect_k$ is the symmetric monoidal category of super vector spaces. This shows that $\sAlgbi$ is a \quot{delooping} of the symmetric monoidal category of super vector spaces, and hence allows to say that $\sAlgbi$ is a bicategory of \textit{super 2-vector spaces}. 
\end{remark}

\begin{remark}
\label{re:dualizable}
One can show that every super algebra $A$ is dualizable with respect to the symmetric monoidal structure. 
Furthermore, a super algebra is \emph{fully} dualizable if and only if it is semisimple (equivalently, separable).
These results are known for ungraded algebras; see, e.g., \cite{Bartlett2015}; the discussion of super algebras is essentially analogous and will appear elsewhere. 
In each case, the dual super algebra is provided by the opposite super algebra $A^{\opp}$, and the evaluation and co-evaluation bimodules $A \otimes A^{\opp} \to k$ and $k \to A^{\opp} \otimes A$ both have $A$ as their underlying super vector space. 
We refer to \cite{Lorand2019} for more details about this duality transformation on the bicategory $\Algbi_k$.
\end{remark}  

Of particular interest are  the \emph{invertible objects}, \ie super algebras $A$ for which evaluation and coevaluation bimodules are invertible.  
The following is a classical theorem, see, for instance, \cite{Wall1964,DK70,Lam2005}.

\begin{theorem}  \label{TheoremCentralSimpleAlgebras}
The invertible objects in $\sAlgbi_k$ are precisely the central simple super algebras.
\end{theorem}

Here, a super algebra $A$ is called \emph{simple}, if $1\neq 0$, and $\{0\}$ and $A$ are the only graded two-sided  ideals of $A$.
A super algebra is called \emph{central}, if the even part of the center is the ground field $k$, \ie $Z(A)_0=Z(A) \cap A_0= k\cdot 1$ (this definition is due to Wall \cite{Wall1964}). 

\begin{remark}
  Freed  defines in \cite{Freed2012} a super algebra to be central when the graded center 
  \begin{equation*}
  SZ(A) := \{ a \in A \mid \forall b \in B : ab = (-1)^{|a||b|} ba \} = k
  \end{equation*}
   (here $[\cdot, \cdot]$ denotes the supercommutator), which seems to be the more natural definition in the graded context.
   The condition $SZ(A) = k$ implies $Z(A)_0 = k$, as any $a \in Z(A)_0$ is also contained in $SZ(A)$.
   On the other hand, the condition $Z(A)_0 = k$ does not generally imply $SZ(A)=k$, as, e.g., the exterior algebra $\Lambda V$ on a $k$-vector space $V$ is supercommutative, hence satisfies $SZ(\Lambda V) = \Lambda V$, but if $\dim(V) = 1$, then $Z(A)_0 = \Lambda^0 V = k \neq \Lambda V$.
   However, it follows a posteriori from \cref{TheoremCentralSimpleAlgebras} and  case by case inspection that every central \emph{simple} algebra satisfies $SZ(A) = k$.
Hence, ultimately, it makes no difference whether one uses the graded or the ungraded center in order to define central simple algebras.
\end{remark}

Central simple super algebras and super algebra homomorphisms form
a full subcategory  $\cssAlg_k \subset \sAlg_k$.
The graded tensor product of central simple super algebras is again central simple; hence, this subcategory is symmetric monoidal.

Likewise, we denote by $\cssAlgbi_k\subset \sAlgbi_k$ the full sub-bicategory over all central simple algebras, which is also symmetric monoidal. \Cref{TheoremCentralSimpleAlgebras} means that
\begin{equation*}
\cssAlgbi_k = (\sAlgbi_k)^{\times}\text{,}
\end{equation*}
where $(..)^{\times}$ denotes the full sub(-bi)category over all invertible objects of a symmetric monoidal (bi)category. 
 The framing $\sAlg_k \to \sAlgbi_k$ of \cref{lem:framingAlg} restricts to a framing  $\cssAlg_k \to \cssAlgbi_k$. 

The \emph{Brauer-Wall group} of the field $k$ is
\begin{equation*}
  \mathrm{BW}_k := \mathrm{h}_0(\sAlgbi_k)^{\times} = \hc 0 \cssAlgbi_k\text{;}
\end{equation*}
here, $\hc 0$ denotes the set of isomorphism classes of objects. Thus, $\mathrm{BW}_k$ is  the set of Morita equivalence classes of central simple super algebras.
It is well known that $\mathrm{BW}_\R=\Z_8$ and $\mathrm{BW}_{\C}=\Z_2$, and that representatives are the real and complex Clifford algebras, $\Cl_n$ ($n=0,...,7$) and $\CCl_n$ ($n=0,1$), respectively.

\begin{remark}
\begin{enumerate}[(1)]
\item
\label{re:classbimodcsa}
If $A$ is a central simple super algebra, then up to isomorphism of $A$-$A$-bimodules there are precisely two invertible $A$-$A$-bimodules, namely $A$ and $\Pi A$, i.e., $\mathrm{Pic}(A)=\Z_2$. Though this is well known, we have been unable to find a direct reference in the established literature. A full proof can be found in \cite[Lem.~2.2.3]{Mertsch2020}.

\item  
There is in fact another symmetric monoidal structure on the bicategory $\sAlgbi_k$ of super algebras, namely the (graded) direct sum. The two monoidal structures are compatible with each other in the sense of distributive laws, and form a \emph{commutative rig bicategory}.

\item
The invertible objects in the bicategory $\Algbi_k$ of \emph{ungraded} algebras are the central simple \emph{ungraded} algebras. The \emph{Brauer group} of the field $k$ is  $\mathrm{Br}_k := \mathrm{h}_0(\Algbi_k)^{\times}$. It is well known that $\mathrm{Br}_{\R}=\Z_2$ and that $\mathrm{Br}_{\C}$ is the trivial group.  

\end{enumerate}
\end{remark} 

\subsection{The automorphism 2-group of an algebra}

\label{sec:aut}

We recall that a \emph{2-group} is a monoidal groupoid in which all  objects are weakly invertible \cite{baez5}. If $\mathscr{C}$ is a bicategory, then every object $A$ of $\mathscr{C}$ has an automorphism 2-group 
\begin{equation*}
\AUT_{\mathscr{C}}(A) \subset \Endcat_{\mathscr{C}}(A) = \Homcat_{\mathscr{C}}(A,A)\text{,}
\end{equation*}
consisting of all \emph{invertible} 1-morphisms $A \to A$ and all invertible 2-morphisms between those. In this section, we discuss the automorphism 2-group of a super algebra $A$ in the bicategory $\sAlgbi_k$.

A 2-group is called \emph{strict} if the monoidal structure is strict and induces group structures on the sets of objects and morphisms. 
Strict 2-groups can be explored via crossed modules, i.e.,  group homomorphisms $t: H \to G$ together with a left action $(g,h) \mapsto \lw gh$ of $G$ on $H$ by group homomorphisms, such that $t(\prescript{g}{}h)=g t(h)g^{-1}$ and $\prescript{t(x)}{}h=xhx^{-1}$ hold for all $g\in G$ and $h,x\in H$.
The relation between crossed modules and strict 2-groups is established by the following construction. To a crossed module we associate the strict 2-group $\Gamma$ with objects $\Gamma_0=G$ and morphisms $\Gamma_1=H \times G$; source and target maps are given by $(h,g)\mapsto g$ and $(h,g)\mapsto t(h)g$, respectively, and composition is defined by $(h_2,g_2) \circ (h_1,g_1) := (h_2h_1,g_1)$. The strict monoidal structure is defined on objects by $(g_1,g_2)\mapsto g_1g_2$ and on morphisms by $(h_2,g_2) \circ (h_1,g_1) := (h_2\lw{g_2}h_1,g_2g_1)$; it induces the given group structure on $G$ and a semi-direct product structure on $H \times G$.

Crossed modules are in particular convenient to model strict 2-groups in a smooth setting. A \emph{crossed module of Lie groups} is one where $H$ and $G$ are Lie groups,  $t$ is a Lie group homomorphism, and where the action of $G$ on $H$ forms a smooth map $G \times H \to H$ \cite{Mackenzie1989}. Every crossed  module  $\Gamma$ of Lie groups gives rise to an exact four-term-sequence
\begin{equation*}
0 \to \pi_1(\Gamma) \to H \stackrel{t}{\to} G \to \pi_0(\Gamma) \to 1
\end{equation*}
of groups, in which $\pi_1(\Gamma) := \mathrm{ker}(t)$ is an abelian Lie group. If $\pi_0(\Gamma):= G/t(H)$ carries a Lie group structure such that $G \to \pi_0(\Gamma)$ is smooth, then $\Gamma$ is called \emph{smoothly separable}. A smoothly separable crossed module $\Gamma$ sits in an extension
\begin{equation}
\label{eq:cmex}
1\to \mathscr{B}\pi_1(\Gamma) \to \Gamma \to \pi_0(\Gamma)_{dis}\to 1
\end{equation}
of crossed modules of Lie groups. 
Here, $\mathscr{B}\pi_1(\Gamma)$ is the  crossed module $\pi_1(\Gamma) \to \ast$ with the trivial action, and $\pi_0\Gamma_{dis}$ is the crossed module $\pi_0(\Gamma) \stackrel{\id}\to \pi_0(\Gamma)$ with the conjugation action.   
The action of $G$ on $H$ induces an action of $\pi_0(\Gamma)$ on $\pi_1(\Gamma)$, and the extension \cref{eq:cmex} is called \emph{central} when this action is trivial.

We will show that the automorphism 2-group of an algebra $A$ in the bicategory $\sAlgbi_k$ is canonically isomorphic to a strict 2-group, and that this strict 2-group even comes from a crossed module of Lie groups.

Let $A$ be a super algebra. We start by introducing the relevant  crossed module of Lie groups, denoted $\AUT(A)$. It consists of the Lie group $A^{\times}_0$ of even invertible elements of $A$ and of the Lie group $\Aut(A)$ of even automorphisms of $A$, together with the Lie group homomorphism $i: A_0^{\times} \to \Aut(A)$ that associates to an element $a\in A_0^{\times}$ the conjugation $i(a)$ by $a$, and the action of $\Aut(A)$ on $A_0^{\times}$ by evaluation. 
The two conditions for crossed modules are evidently satisfied, 
\begin{equation*}
i({\varphi(a)})=\varphi \circ i(a) \circ \varphi^{-1}
\quad\text{ and }\quad
i(a)(b)=aba^{-1}\text{.}
\end{equation*}
The four-term-sequence associated to $\AUT(A)$ is
\begin{equation} \label{4TermSequence}
1\to Z(A)_0^{\times} \to A_0^{\times} \stackrel{i}{\to} \Aut(A) \to \mathrm{Out}_0(A) \to 1\text{,}
\end{equation}
where $\mathrm{Out}_0(A)$ is by definition the quotient $\Aut(A)/A_0^{\times}$. In many cases, $\Out_0(A)$ can be equipped with a Lie group structure showing that $\AUT(A)$ is smoothly separable; then, we have an extension of crossed modules of Lie groups,
\begin{equation}
\label{eq:exactsequenceAUTA}
1\to \mathscr{B}Z(A)_0^{\times} \to \AUT(A) \to \mathrm{Out}_0(A)_{dis} \to 1\text{.}
\end{equation} 
Since the  action 
of $\mathrm{Out}_0(A)$ on $Z(A)^{\times}_0$, given by $[\phi]\cdot z := \phi(z)$, 
is in general non-trivial,   $\AUT(A)$ is in general a \emph{non-central} extension of crossed modules.

We let $\AUT(A)^{\delta}$ be the \emph{discrete} strict 2-group corresponding to $\AUT(A)$, obtained by discarding the smooth  structures and performing the transition to strict 2-groups as explained above.
We consider the functor
\begin{equation}
\label{eq:autaut}
\AUT(A)^{\delta} \to \Endcat_{\sAlgbi_k}(A)
\end{equation}
that associates to an object $\phi\in \Aut(A)$ the $A$-$A$-bimodule $A_{\phi}$, and to a morphism $(a,\phi)\in A_0^{\times} \times \Aut(A)$ the intertwiner $r_a: A_{\phi} \to A_{i(a) \circ \phi}$ from \cref{lem:framing:d}. 
It is straightforward to show that this defines a functor and that it respects the monoidal structures.

\begin{proposition}
\label{prop:aut}
Let $A$ be a Picard-surjective super algebra. Then, the functor \cref{eq:autaut} establishes an equivalence of strict 2-groups between $\AUT(A)^{\delta}$ and the automorphism 2-group  of $A$ as an object in the bicategory $\sAlgbi_k$,
\begin{equation*}
\AUT(A)^{\delta} \cong \AUT_{\sAlgbi_k}(A)\text{.}
\end{equation*} 
\end{proposition}

\begin{proof}
That the image of the functor \cref{eq:autaut} lies in the subcategory $\AUT_{\sAlgbi_k}(A)$ is proved by \cref{LammaIdentityBimoduleInvertible,lem:framing:d}. Now it remains to show that the co-restriction of the functor \cref{eq:autaut}  to $\Aut_{\sAlgbi_k}(A)$ is an equivalence of categories; it is then automatically an equivalence of monoidal categories. Picard-surjectivity of $A$ implies essential surjectivity, and \cref{lem:framing:d} proves that  it is full and faithful.
\end{proof}

Independently of whether or not $\Out_0(A)$ can be equipped with a Lie group structure, a \emph{discrete} (i.e., non-smooth) version of the extension \cref{eq:exactsequenceAUTA} always exists. Thus, \cref{prop:aut} implies the following. 

\begin{corollary}
Let $A$ be a Picard-surjective super algebra. 
The automorphism 2-group of $A$ fits into an extension of  2-groups
\begin{equation*} 
1 \to \mathscr{B}Z(A)_0^{\times} \to \AUT_{\sAlgbi_k}(A) \to \Out_0(A)_{dis} \to 1\text{.}
\end{equation*}
\end{corollary}

The automorphism 2-groups of isomorphic objects of a bicategory are always isomorphic.
By \cref{prop:picsur}, every super algebra $A$ is Morita equivalent to a Picard-surjective super algebra $B$, i.e., $A \cong B$ in the bicategory $\sAlgbi_k$. In particular, $\AUT_{\sAlgbi_k}(A)\cong \AUT_{\sAlgbi_k}(B)$. 
Hence, the result of \cref{prop:aut} extends non-canonically from Picard-surjective super algebras to arbitrary ones, as follows.

\begin{corollary}
Let A be a super algebra, and let $B$ be a Morita equivalent Picard-surjective super algebra. Then, there exists an isomorphism $\AUT_{\sAlgbi_k}(A)\cong\AUT(B)^{\delta}$.
\end{corollary}

 Moreover,  we see that $\AUT(A)^{\delta} \cong \AUT(B)^{\delta}$  whenever $A$ and $B$ are Picard-surjective and Morita equivalent super algebras. 
Our next result shows that this also holds  in the smooth setting.

\begin{proposition}
\label{prop:moritainvariance2group}
Suppose $A$ and $B$ are Picard-surjective super algebras. Then, every invertible $A$-$B$-bimodule induces a weak equivalence between the crossed modules of Lie groups $\AUT(A)$ and $\AUT(B)$.
\end{proposition}

\begin{proof}
Weak equivalences between crossed modules of Lie groups can be modelled by invertible butterflies \cite{Aldrovandi2009}.
Let $M$ be an invertible $A$-$B$-bimodule, and let $I(M)$ be the Lie group of implementers that will be defined in \cref{DefinitionImplementerSet} below. 
We claim that the diagram 
\begin{equation} \label{ButterflyMoritaEquivalence}
\xymatrix{A_0^{\times} \ar[dd]_{i} \ar[dr]^{i_\ell} && B_0^{\times} \ar[dd]^{i}\ar[dl]_{i_r} \\ & I(M) \ar[dl]^{p_\ell}\ar[dr]_{p_r} \\ \mathrm{Aut}(A) && \Aut(B)\text{,}}
\end{equation}
of Lie groups and Lie group homomorphisms is an invertible butterfly between the crossed modules $\AUT(A)$ and $\AUT(B)$. 
In \eqref{ButterflyMoritaEquivalence}, $p_\ell$ and $p_r$ are the projection maps as in \eqref{SourceTargetMaps}.
The map $i_\ell : A_0^{\times} \to I(M)$ is defined by sending $a \in A_0^{\times}$ to $(i(a),l_a, \id_B)$, where $i(a)$ is conjugation by $a$, and $l_a : M \to M$, $m \mapsto a \lact m$ is the left  action by $a$. 
Similarly, $i_r : B_0^\times \to I(M)$ is defined by sending $b\in B_0^{\times}$ to $(\id_A,r_{b^{-1}}, i(b))$, where $r_{b}: M \to M$, $m \mapsto m \ract b$, is the right action. 
%
%
One easily checks that $i_\ell$ and $i_r$ are group homomorphisms.
\\
In order to prove that \eqref{ButterflyMoritaEquivalence} is a butterfly, we note that the triangular diagrams in \eqref{ButterflyMoritaEquivalence} commute and that the diagonal sequences complexes. 
We also have to check the following two identities, which is straightforward:
\begin{equation*}
 (i(\phi(a)), l_{\phi(a)}, \id_B)  = (\phi, u, \psi)(i(a), l_a, \id_B)(\phi, u, \psi)^{-1}
\end{equation*}
and
\begin{equation*}
  (\id_A, r_{\psi(b)^{-1}}, i(\psi(b))) = (\phi, u \psi) (\id_A, l_{b^{-1}}, i(b)) (\phi, u, \psi)^{-1} \text{.}
\end{equation*}
Then it remains to be shown that the diagonal sequences are short exact. 
For injectivity of $i_\ell$ and $i_r$, we note that, as a Morita equivalence, the bimodule $M$ is faithfully balanced, \ie the maps $A \to \mathrm{End}_{-,B}(M)$, $a \mapsto l_a$ and $B^{\opp} \to \mathrm{End}_{A,-}(M)$, $b \mapsto r_b$, are algebra isomorphisms. 
Thus, we have $l_a=\id_A$ if and only if $a=1$, and similarly $r_b=\id_B$ if and only if $b=1$; this shows injectivity. 
Surjectivity follows from the Picard-surjectivity and \cref{RemarkBetterThanImplementingPicardSurjective}.
It is left to show exactness in the middle. 
Suppose $(\phi, u, \psi)\in I(M)$ and $\phi=\id_A$. We then obtain an even invertible intertwiner
\begin{equation*}
\xymatrix{B \cong M^{-1} \otimes_A M \ar[r]^-{1\otimes_A u}&  M^{-1}\otimes_A M_{\psi}\cong (M^{-1}\otimes M)_{\psi}\cong B_{\psi}}
\end{equation*}
of $B$-$B$-bimodules. By \cref{lem:framing:d}, this shows that $\psi$ is inner, \ie $\psi=i(b)$ for some $b\in B_0^{\times}$, and that above intertwiner is given by $x \mapsto xb^{-1}$. This shows that $u=r_{b^{-1}}$ and thus that $i_r(b)=(\id_A,u,\psi)$. The exactness of the other sequence is proved in the same way. 
\end{proof}

Finally, we want to compute the automorphism 2-group of a (Picard-surjective) central simple super algebra.
A central simple super algebra has by definition $Z(A)^{\times}_0=k^{\times}$, and $\mathrm{Pic}(A)=\Z_2$, with representatives given by $A$ and $\Pi A$, see \cref{re:classbimodcsa}. 
We consider the  crossed module of Lie groups $k^{\times} \to \Z_2$, which is given by the zero map and the trivial action of $\Z_2$ on $k^{\times}$. 
It can be written as a direct product crossed module $\mathscr{B}k^{\times} \times (\Z_2)_{dis}$.

\begin{proposition}
\label{prop:classcssa}
If $A$ is a Picard-surjective central simple super algebra over $k$, then there exists a canonical weak equivalence $\AUT(A) \cong \mathscr{B}k^{\times} \times (\Z_2)_{dis}$. 
\end{proposition}

\begin{proof}
        In order to prove the weak equivalence, we provide a butterfly
\begin{equation} 
\label{Butterflyclasscssa}
\begin{aligned}
        \xymatrix{k^{\times} \ar[dd]_{0} \ar[dr]^{i_1} && A_0^{\times} \ar[dl]_{i_2}  \ar[dd]^{c} \\ & K \ar[dr]_{p_2} \ar[dl]^{p_1} & \\ \Z_2  && \Aut(A) \ .}
\end{aligned}
\end{equation}
We define $K$ as the subset of $\Z_2 \times \underline{\mathrm{GL}}(A) \times \Aut(A)$ consisting of elements $(\epsilon, u ,\phi)$ where $\phi\in \Aut(A)$, $\epsilon\in \Z_2$, $u: \epsilon A \to A_{\phi}$ is an even invertible intertwiner of $A$-$A$-bimodules. 
Here we set
\begin{equation*}
\epsilon A = \begin{cases}A & \epsilon=0 \\
\Pi A & \epsilon=1\text{.}\ \\
\end{cases}
\end{equation*}
It is straightforward to see (cf.~\cref{lem:framing:d}) that for $\phi$ fixed and $\epsilon=0$, the possible intertwiners $u$ such that $(\epsilon, u, \phi) \in K$ correspond precisely to elements $a\in A_0^{\times}$ such that $u(x)=xa^{-1}$ and $\phi=i(a)$. Similarly, for $\epsilon=1$, possible intertwiners correspond precisely to elements $a\in A_1^{\times}$ such that $u(x)=xa^{-1}$ and $\phi=\eta \circ i(a)$, where $\eta$ is the grading operator of $A$.
%
%
\,\\
The subset $K$ is in fact a subgroup: if $(\epsilon_1, u ,\phi_1)$ and $(\epsilon_2,u_2, \phi_2)$ are in $K$, then $(\epsilon_1 + \epsilon_2,u_1 \circ u_2, \phi_1\circ\phi_2)$ is also in $K$ since $u_1 \circ u_2 \in \underline{\mathrm{GL}}(A)$ indeed is an intertwiner $(\epsilon_1 + \epsilon_2) A \to A_{\phi_1\circ \phi_2}$. 
To see this, we note that if the $u_i$ correspond to $a_i\in A^{\times}_{\epsilon_i}$, then $u_1 \circ u_2$ corresponds to $a_1a_2\in A^{\times}_{\epsilon_1 + \epsilon_2}$. 
Since the parity operator commutes with the even automorphism $i({a_1})$, we have $\phi_1 \circ \phi_2 =  \eta \circ i({a_1})\circ i({a_2})$ if $\epsilon_1+\epsilon_2=1$ and $\phi_1\circ \phi_2 = i({a_1})\circ i({a_2})$ otherwise. 
An argument similar to the one in \cref{sec:implementingbimodules} shows that
 $K$ is even a Lie subgroup of the Lie group $\Z_2 \times \underline{\mathrm{GL}}(A) \times \Aut(A)$.
\\
The group homomorphism $i_1: k^{\times} \to K$ and $i_2: A_0^{\times} \to K$ are defined by
\begin{equation*}
  i_1(\lambda) = (0,\lambda,\id), \qquad i_2(a) = (0,r_{a^{-1}},i(a)) \text{,}
\end{equation*}
where $r_a: A \to A$ is right multiplication by $a$. The group homomorphisms $p_1: K \to \Z_2$ and $p_2: K \to \mathrm{Aut}(A)$ are the projections.
It is then obvious that the triangles in \eqref{Butterflyclasscssa} are commutative and that  the diagonal maps are complexes. The two conditions for butterflies, which read here as
\begin{equation*}
(0,\lambda, \id) = (\epsilon,u,\phi)(0,\lambda, \id)(\epsilon, u, \phi)^{-1}
\end{equation*}
for all $\lambda\in k^{\times}$ and
\begin{equation*}
(0,r_{\phi(a)^{-1}}, i({\phi(a)}))=(\epsilon,u, \phi)(0,r_{a^{-1}},i(a))(\epsilon,u,\phi)^{-1}
\end{equation*}
for all $a\in A_0^{\times}$, are straightforward to check.
\\
It remains to prove the exactness of the two sequences. 
It is clear that $i_1$ and $i_2$ are injective. 
Now consider $(0,u,\phi) \in \ker(p_1)$. 
As explained above, $u$ corresponds to an element $a\in A_0^{\times}$ with $u=r_{a^{-1}}$ and $\phi=i(a)$, hence $(0, u, \phi) = i_2(a) \in \mathrm{im}(i_2)$. 
Next consider $(\epsilon, u, \id) \in \ker(p_2)$. 
Then $u : \epsilon A \to A$ is an invertible even intertwiner, \ie isomorphism of $A$-$A$-bimodules.
However, because $\mathrm{Pic}(A) \cong \Z_2$, generated by $[\Pi A]$, see \cref{re:classbimodcsa}, $\Pi A$ and $A$ are not isomorphic as $A$-$A$-bimodules, hence it follows that $\epsilon=0$. 
This shows exactness of both sequences in the middle. 
Since $A$ is Picard surjective, there exists an even invertible intertwiner $u: \Pi A \to A_{\phi}$, for some $\phi\in \Aut(A)$. Thus, $(1, u, \phi)\in K$ maps to $1\in \Z_2$; this shows the surjectivity of $p_1$. Given $\phi\in \mathrm{Aut}(A)$, we have that $A_{\phi} \in \mathrm{Pic}(A)$ and hence either $A_{\phi} \cong A$ or $A_{\phi} \cong  \Pi A$. This shows the surjectivity of $p_2$.
 \end{proof}

Since every central simple super algebra is Morita equivalent to a Picard-surjective  one (\cref{prop:picsur}), which is still central and simple, we obtain the following. 

\begin{corollary}
If $A$ is a central simple super algebra, then there exists a weak equivalence of 2-groups
\begin{equation*}
\AUT_{\sAlgbi_k}(A) \cong \mathscr{B}k^{\times} \times (\Z_2)_{dis}\text{.}
\end{equation*}
\end{corollary}

\section{The bicategory of implemented algebras}
\label{SectionImplementingModules}

In this section, we set up a sub-bicategory $\sAlgbiimp_k$ of the bicategory $\sAlgbi_k$ of algebras discussed in \cref{sec:bicatalgebras}, whose objects are all super algebras, but whose morphisms are only so-called \emph{implementing} bimodules. Our motivation is that only the sub-bicategory $\sAlgbiimp_k$ admits a bundle version. On the other hand,   neither the framing nor the symmetric monoidal structure restrict to this sub-bicategory.

\subsection{Implementing bimodules}

\label{sec:implementingbimodules}

We define a category $\Impcat_k$ whose  objects are triples $(A,M,B)$ consisting of super algebras $A$ and $B$ and an $A$-$B$-bimodule $M$, and whose morphisms $(A,M,B) \to (A',M',B')$ are triples $(\phi,u,\psi)$ where $\phi: A \to A'$ and $\psi: B \to B'$ are super algebra homomorphisms, and $u: M \to M'$ is an (even) intertwiner along $\phi$ and $\psi$, \ie it satisfies the relation \eqref{IntertwiningCondition}. 
Equivalently,  $u$ is an intertwiner $M \to \pss{\phi}M^\prime_\psi$ of $A$-$B$-bimodules. A further equivalent point of view is that $u$ \emph{implements} $\phi$ and $\psi$ in the sense that
\begin{equation} \label{ImplementingRelation}
\phi(a) \lact m = u^{-1}(a \lact u(m)), \qquad m \ract \psi(b) =  u^{-1}(u(m) \ract b), \qquad a \in A, ~~m \in M, ~~b \in B\text{.}
\end{equation}
Of particular interest will be the automorphism groups of objects of the category $\Impcat_k$.

\begin{definition}[Implementers] \label{DefinitionImplementerSet}
Let $A$ and $B$ be super algebras, and let $M$ be an $A$-$B$-bimodule. The automorphism group 
\begin{equation*}
I(M) := \Aut_{\Impcat_k}(A,M,B)
\end{equation*}
is called the \emph{group of implementers} of $M$. 
\end{definition}

\begin{remark}
\begin{enumerate}[(1)]

\item 
For $(\phi,u,\psi) \in I(M)$, the data of $\phi$ and $\psi$ are redundant in the case that $A$ and $B$ act faithfully on $M$: in this case they are determined by \eqref{ImplementingRelation}.
This is false in general when the annihilators $\mathrm{Ann}(M) \subset A$ and $\mathrm{Ann}(M) \subset B$ are non-trivial.

\item
\label{re:twistedbimoduleimplementers}
As in any category, isomorphisms $(A,M,B)\cong (A',M',B')$ between objects induce isomorphisms $I(M)\cong I(M')$ of automorphism groups by conjugation. 
In particular, since for any automorphisms $\phi\in \Aut(A)$ and $\psi\in \Aut(B)$, the triple
 $(\phi,\id_M,\psi)$ is an isomorphism
$(A,\pss{\phi}M_\psi,B) \cong (A,M,B)$ we have $I(M) \cong I(\pss{\phi}M_\psi)$. 

\item
\label{re:enlargement}
The group $I(M)$ can be seen as an enlargement   of the group  $\mathrm{Aut}_{A\text{-}B}(M)$ of invertible bimodule automorphisms, via the embedding $\mathrm{Aut}_{A\text{-}B}(M) \to I(M)$ sending $u$ to $(\id_A,u,\id_B)$. 
This enlargement  is essential when defining bimodule bundles, as it turns out that the transition functions of local trivializations take values in $I(M)$ and cannot be required to take values in the smaller group $\Aut_{A\text{-}B}(M)$ without resulting in a trivial theory.

\end{enumerate}

\end{remark}

It is easy to check that $I(M)$ is a subgroup of the Lie group $\mathrm{Aut}(A) \times \GL(M) \times \mathrm{Aut}(B)$.
To see that it is a closed subgroup, hence a Lie group itself, consider for $a\in A$, $b\in B$, $m \in M$, the continuous map 
\begin{equation*}
g_{a,b,m}: \mathrm{Aut}(A) \times \mathrm{Aut}(B) \times \GL(M) \to M, \qquad (\phi, u, \psi) \mapsto u(a\lact m \ract b) - \phi(a)\lact u(m) \ract \psi(b).
\end{equation*}
Then $I(M) = \bigcap_{a,b, m} g_{a,b, m}^{-1}(\{0\})$ is an intersection of closed sets, hence closed. 
The projection maps
\begin{equation} \label{SourceTargetMaps}
  p_\ell: I(M) \to \Aut(A), \qquad p_r: I(M) \to \Aut(B)
\end{equation} 
are smooth group homomorphisms. 

\begin{definition}[Implementing modules] \label{DefinitionImplementing}
Let $A$ and $B$ be super algebras and let $M$ be an $A$-$B$-bimodule.
Then, $M$ is called \emph{left-implementing} if $p_\ell$ is open, \emph{right implementing} if $p_r$ is open and \emph{implementing} if it is both left and right implementing.
\end{definition}

\begin{remark}
A few remarks are in order.
\begin{enumerate}[(1)]

\item \label{RemarkIdentityComponentsImplementing}
The condition that $p_\ell$ is open is equivalent to the condition that the restriction $p_\ell: I(M)_0 \to \Aut(A)_0$ to identity components is surjective. 
(We emphasize here that the notation $G_0$ stands for the identity component of a Lie group $G$ and should not be confused with notation $A_0$ for the even degree part of a super algebra $A$.)
In particular, it follows from these equivalent conditions that the restriction $p_{\ell}: I(M)_0 \to \Aut(A)_0$ is a submersion, and in turn, that the whole map $p_{\ell}: I(M) \to \Aut(A)$ is a submersion. The same remarks hold for $p_r$.

\item
It follows from the previous remark that any automorphism $\phi \in \Aut(A)_0$ can be implemented on any left implementing $A$-$B$-bimodule $M$, in the sense that there exists an invertible linear map $u :M \to M$ such that $\phi(a) \lact m = u^{-1}(a \lact u(m))$.
Similarly, if $M$ is right implementing, then any automorphism $\psi \in \Aut(B)_0$ can be implemented on $M$ in the sense that there exists $v$ such that $m \lact \psi(b) =  v^{-1}(v(m) \lact b)$.
If $M$ is (both left and right) implementing, then both $u$ and $v$ exist.
However, in general we cannot choose $u = v$, since the map $p_\ell \times p_r : I(M) \to \Aut(A)_0 \times \Aut(B)_0$ is \emph{not} necessarily surjective.

\item
In \cref{re:openinsteadofsurjective} we explain why we do not require that $p_\ell$ and $p_r$ are surjective.

\item
\label{RemarkBetterThanImplementingPicardSurjective}
If super algebras $A$ and $B$ are Picard-surjective (see \cref{DefinitionPicardSurjective}), then for an invertible $A$-$B$-bimodule, the maps $p_\ell : I(M) \to \Aut(A)$ and $p_r : I(M) \to \Aut(B)$ are surjective without restricting to the identity component.
To see this, let $\phi \in \Aut(A)$. 
Then $M^{-1} \otimes_A \prescript{}{\phi^{-1}}M$ is an invertible $B$-$B$-bimodule, and hence, by Picard-surjectivity of $B$, there exists $\psi\in\Aut(B)$ and an even invertible intertwiner $M^{-1} \otimes_A \pss{\phi^{-1}}M\cong B_{\psi}$. Tensoring with $M$ from the left, we obtain an even invertible intertwiner $u: \pss{\phi^{-1}}M \to M \otimes_B B_{\psi}\cong M_{\psi}$. 
Thus, $(\phi, u, \psi)$ is a preimage of $\phi$ under $p_\ell$. 
Surjectivity of $p_r$ is proved in the same way. 

\item \label{RemarkTwistedModulesImplementing}
Suppose $(\phi,u,\psi)$ is an isomorphism $(A,M,B)\cong (A',M',B')$ in $\Impcat_k$.
The corresponding  isomorphism $I(M) \cong I(M')$ is given by $(\alpha, v, \beta) \mapsto (\phi \alpha  \phi^{-1}, uvu^{-1}, \psi \beta  \psi^{-1})$. Since the conjugation action on a Lie group preserves the identity component, $M$ is (left/right) implementing if and only if $M'$ is (left/right) implementing.
In particular, using \cref{re:twistedbimoduleimplementers}, an $A$-$B$-bimodule $M$ is (left/right) implementing if and only if $\pss{\phi}M_\psi$ is (left/right) implementing, for any automorphisms $\phi\in \Aut(A)$ and $\psi\in \Aut(B)$.

\item
It is straightforward to show that the  direct sum of implementing bimodules is implementing.

\end{enumerate}
\end{remark}

\begin{example}
We give some first examples of implementing and non-implementing bimodules.
\begin{enumerate}[(1)]
\item 
 \label{ExampleImplementing1}
For any super algebra $A$, the identity bimodule $A$ is implementing: Any automorphism $\phi \in \Aut(A)$ is both left and right implemented by $\phi^{-1}$.
It follows from \cref{RemarkTwistedModulesImplementing} that also the $A$-$A$-bimodules $\pss{\phi}A_\psi$ are implementing, for any $\phi, \psi \in \Aut(A)$.

\item \label{ExampleImplementingExteriorTensorProductWithVB}
If $M$ is an $A$-$B$-bimodule and $E$ is a super vector space, then the exterior tensor product $M \otimes_k E$ is again an $A$-$B$-bimodule, with bimodule structure
\begin{equation*}
  a \lact (m \otimes v) \ract b = (-1)^{|b||v|} (a \lact m \ract b) \otimes v, \qquad a \in A, ~~b \in B, ~~ m \in M, ~~ v \in E.
\end{equation*}
If $M$ is implementing, then $M \otimes_k E$ is again implementing: If $(\phi, u, \psi) \in I(M)$, then $(\phi, u \otimes_k \id_E, \psi) \in I(M \otimes_k E)$.

\item \label{ExampleImplementing2}
        Let $A = \Lambda k$, the exterior algebra on one generator $\epsilon$, with $\epsilon^{2}=0$.
We have $\Aut(A) \cong k^\times = \GL_1(k)$, where the automorphism $\varphi_\lambda$ corresponding to $\lambda \in k^\times$ is given by $\varphi_\lambda(a + b\epsilon) = a + \lambda b  \epsilon$.
The module $M = \Lambda^1 k = \{ a \epsilon \mid a \in k\}$ is implementing both as an $A$-$A$-bimodule and as an $A$-$k$-bimodule: as $\pss{\phi} M_{\psi} = M$ for all automorphisms $\phi, \psi \in A$,  any pair of automorphisms of $A$ is implemented on $M$ by the identity map.

\item \label{ExampleImplementing3}
Let $A = \Lambda k^2 = \Lambda k \otimes \Lambda k$, the exterior algebra on two generators $\epsilon_1$ and $\epsilon_2$.
We claim that the $A$-$A$-bimodule $M = \Lambda k \otimes \Lambda^1 k$ is \emph{not} left implementing.
To see this, observe that $\mathrm{GL}_2(k)$ acts on $A$ by automorphisms, by functoriality of the exterior algebra, and that this action is faithful.
Let $\phi \in \Aut(A)$ be the automorphism associated to the linear map sending $\epsilon_1 \mapsto -\epsilon_2$, $\epsilon_2 \mapsto \epsilon_1$, which is contained in $\mathrm{GL}_2(k)_0 \subset \Aut(A)_0$. 
Suppose that $u$ acts implements $\phi$, then we have
\begin{equation} \label{wowefojn}
  \phi(\epsilon_2) \lact m = u^{-1} (\epsilon_2 \lact u(m)) = 0, \qquad \forall m \in M,
\end{equation}
since $\epsilon_2$ acts trivially on $M$.
However, $\phi(\epsilon_2) = \epsilon_1$ does not act non-trivially on $M$, hence the left hand side of \eqref{wowefojn} is not always zero, a contradiction.
One shows similarly that $M$ is also not right implementing.

\item
\label{ExampleTwistedModuleImplementing}
We analyse when the $B$-$A$-bimodule $B_\varphi$ is implementing, for $\varphi:A \to B$ a (not necessarily invertible) super algebra homomorphism, in the case that $B = M_{n}(k)$.
As for any super algebra $B$, there is an injection $B \to \End(B)$, the space of vector space endomorphisms of $B$, given by $b \mapsto \ell_b$, \ie by sending $b$ to the left multiplication by $b$.
By definition, elements $(\phi, u, \psi) \in I(B_\varphi)$ satisfy $u(b) = u(b \lact 1) = \phi(b) \lact u(1)$ for each $b \in B$.
On the other hand, for $B = M_n(k)$, it is well known that any $\phi \in \Aut(B)$ is inner, hence $\phi = i(b)$ for some $b \in B$, which yields an element $(\phi = i(b), \ell_b, \id) \in I(B_\varphi)$.
Taking the quotient of the two elements in consideration, we obtain that also $(\id, \ell_{b^{-1}} u, \psi) \in I(B_\varphi)$. 
In other words, $\ell_{b^{-1}} u$ commutes with left multiplication by $b \in B$, hence $\ell_{b^{-1}} u$ is right multiplication by some element $b^\prime \in B^\times$.
This discussion shows that $B_\varphi$ is implementing if and only if for each $\psi \in \Aut(A)_0$, there exists an element $b^\prime \in B^\times$ such that 
\begin{equation} \label{ImplementingConditionExample}
i(b^\prime)^{-1} \circ \varphi = \varphi \circ \psi \text{.}
\end{equation}
\item  
\label{ExampleTwistedModuleNotImplementing}
We now give a concrete example where $B_\varphi$ is not implementing in the situation of \cref{ExampleTwistedModuleImplementing}, inspired by \cite{Cornulier}. Suppose that $n = 2m$ is even and let $V$ be some $k$-vector space.
Let $A = k \oplus V$, where the product is given by $(\lambda, v) \cdot (\mu, w) = (\lambda \mu, \lambda w + \mu v)$, in other words, the product of any two elements in $V$ is zero.
Any vector space automorphism $a \in \mathrm{GL}(V)$ gives rise to an algebra automorphism $\psi_a \in \Aut(A)$ defined by $\psi_a(\lambda, v) = (\lambda, a(v))$.
Moreover, any linear map $\tilde{\varphi} : V \to M_m(k)$ gives rise to an algebra homomorphism $\varphi: A \to B = M_n(k)$, given by
\begin{equation*}
  \varphi(\lambda, v) = \begin{pmatrix} \lambda & 0 \\ \tilde{\varphi}(v) & \lambda \end{pmatrix},
\end{equation*}
where each block is a $k \times k$ matrix.
It is now easy to construct examples of $\tilde{\varphi}$ where for some $\psi = \psi_a$, there exists no $b^\prime \in B^\times$ with \eqref{ImplementingConditionExample}.
For example, if $\tilde{\varphi}$ is an isomorphism and both $\dim(V)$ and $m$ are at least $2$, then the image of $\tilde{\varphi}$ contains matrices of different rank, say $\tilde{\varphi}(v_1)$ has rank $r_1$ and $\tilde{\varphi}(v_2)$ has rank $r_2$, where $r_1 \neq r_2$.
Now as $v_1$ and $v_2$ are necessarily linearly independent, there exists $a \in \mathrm{GL}(V)$ exchanging the two.
Then $(\varphi \circ \psi_a)(0, v_1) = \varphi(0, v_2)$ has rank $r_2$, but $\varphi(0, v_1)$ has rank $r_1$.
But condition \eqref{ImplementingConditionExample} requires these elements to be conjugate in $M_n(k)$, which cannot be true as conjugate matrices have the same rank.

\end{enumerate}
\end{example}

Next we provide three systematic results about implementability. The first shows that the dual of an implementing module is again implementing. 
For an $A$-$B$-bimodule $M$, the \emph{dual bimodule} is the vector space $\underline{\Hom}_{k\text{-}B}(M, B)$ of not necessarily grading-preserving linear maps $M \to B$ that intertwine the right $B$-actions.
It is graded by declaring parity-preserving linear maps to be even and parity-reversing linear maps to be odd.
Moreover, it obtains the structure of a $B$-$A$-bimodule by setting
\begin{equation*}
  (b \lact \xi \ract a)(m) = b \lact \xi(a \lact m), \qquad a \in A, ~~m \in M, ~~ b \in B.
\end{equation*}

\begin{proposition} \label{PropAdjointImplementing}
If $M$ is a (left/right) implementing $A$-$B$-bimodule, then the dual module $\underline{\Hom}_{k\text{-}B}(M, B)$ is a (right/left) implementing $B$-$A$-bimodule.
\end{proposition}

\begin{proof}
Let $(\phi, u, \psi) \in I(M)$.
Define $\tilde{u} : M \to B$ by $\tilde{u}(\xi) = \psi \circ \xi \circ u^{-1}$.
A straightforward calculation shows that $\tilde{u}(\xi) \in \underline{\Hom}_{k\text{-}B}(M, B)$.
Its inverse is given by $\tilde{u}^{-1} = \psi^{-1} \circ \xi \circ u$, hence $\tilde{u}$ is an invertible linear transformation of $\underline{\Hom}_{k\text{-}B}(M, B)$.
Another calculation shows that $(\psi, \tilde{u}, \phi) \in I(\underline{\Hom}_{k\text{-}B}(M, B))$.
This implies the proposition.
\end{proof}

The next result assures that by requiring implementability we do not lose any invertible bimodules.

\begin{proposition} \label{PropositionInvertibleImplementing}
Let $A$ and $B$ be super algebras.
Then, any invertible $A$-$B$-bimodule $M$ is implementing.
\end{proposition}

\begin{proof}
Since $M$ is invertible, it is finitely generated and projective as a right $B$-module. 
Therefore, as a right $B$-module, it has the form $M = p (B^r \oplus \Pi B^s)$ for some numbers $r, s \in \N$ and an even idempotent $p \in M_n(B)$, where $n = r+s$.

To show surjectivity of $p_r$, we need the following general lemma.

\begin{lemma} \label{LemmaCloseProjections}
Let $B$ be a super algebra and let $p, q \in B$ be two idempotents.
If for some submultiplicative norm on $B$, we have $\|p-q\| < 1/2$, then there exists an invertible element $b \in B$ such that $q = bpb^{-1}$.
\end{lemma}

\begin{proof}
Set $b = qp + (1-q)(1-p)$. 
Because
\begin{equation*}
  \|1 - b\| = \|q + p - 2qp\| = \|(1-q)(p-q) - q(p-q)\| \leq 2 \|p-q\| < 1,
\end{equation*}
the element $b$ is invertible, with $b^{-1} = \sum_{k=0}^\infty (1-b)^k$, where the series converges absolutely owing to the norm bound on $1-b$ and the fact that as $B$ is finite-dimensional, it is complete.
One calculates that  $qb = qp = bp$, hence $q = bpb^{-1}$, as claimed.
\end{proof}

\noindent Now let $\psi \in \Aut(B)_0$ and let $(\psi_s)_{s \in [0, 1]}$ be a continuous path in $\Aut(B)_0$ with $\psi_0 = \id_B$ and $\psi_1 = \psi$.
Let $0  = s_0 < s_1 < \dots < s_N = 1$ be a partition of the interval $[0, 1]$ such that we have $\|\psi^{(n)}_{s_k}(p) - \psi^{(n)}_{s_{k-1}}(p)\| < 1/2$, where $\psi^{(n)}_s$ is the automorphism of $M_n(B)$ induced by $\psi_t$.
By \cref{LemmaCloseProjections}, there exist $b_k \in M_n(B)$ such that $\psi_{k}(p) = b_k \psi_{k-1}(p) b_k^{-1}$.
Setting $b = b_1 \cdots b_N$, we obtain $b \psi(p) b^{-1} = p$.
 We now consider the vector space automorphism $\tilde{u}$ of $B^r \oplus \Pi B^s$ given by
\begin{equation*}
  \tilde{u}(v) = b\psi^{(n)}(v).
\end{equation*}
We claim that $\tilde{u}$ restricts to a vector space automorphism $u$ of $M \subset B^r \oplus B^s$ that implements $\psi$.
Indeed, for $pv \in M$, we have
\begin{equation*}
  \tilde{u}(pv) = b\psi^{(n)}(pv) = b \psi^{(n)}(p) b^{-1} b \psi^{(n)}(v) = p  \psi^{(n)}(v) \in M,
\end{equation*}
hence $\tilde{u}$ preserves $M$.
Now for $a \in M_n(B)$ and $b \in B$,
\begin{align*}
  u(a \lact pv \ract b) &= b \psi^{(n)}(a \lact pv \ract b) \\
  &= b\psi^{(n)}(a) b^{-1} \lact b\psi^{(n)}(p)\psi^{(n)}(v) \ract \psi(b) \\
  &= b \psi^{(n)}(a) b^{-1} \lact p b\psi^{(n)}(v) \ract \psi(b) \\
  &= b \psi^{(n)}(a) b^{-1} \lact p \tilde{u}(v) \ract \psi(b) \\
  &= b \psi^{(n)}(a) b^{-1} \lact u(v) \ract \psi(b).
\end{align*}
Hence $u$ implements $\psi$.
Now, again since $M$ is invertible, it is faithfully balanced and we can identify $A = \underline{\End}_{k\text{-}B}(M) = p M_n(B) p$, the super algebra of (not necessarily grading preserving) linear maps on $M$ commuting with the right $B$-action, acting on $M$ by left multiplication.
It is clear that if $a \in A$, \ie $a$ commutes with the right multiplication with elements of $B$, then also $b \psi^{(n)}(a) b^{-1}$ commutes with the right multiplication by elements of $B$, hence is contained in $A$.
Therefore $\phi(a) = b \psi^{(n)}(a) b^{-1}$ defines an automorphism of $A$.
In total, we obtain that $(\phi, u, \psi) \in I(M)$, hence $M$ is right implementing.
\\
The proof that $M$ is left implementing is completely analogous.
\end{proof}

\begin{corollary}
For any super algebra $A$, the inclusion $\sAlgbiimp_k \to \sAlgbi_k$ induces an isomorphism of 2-groups $\AUT_{\sAlgbiimp_k}(A) \cong \AUT_{\sAlgbi_k}(A)$. In particular, the Lie 2-group $\AUT(A)$ defined in \cref{sec:aut} is a smooth version of the automorphism group of $A$ in $\AUT_{\sAlgbiimp_k}(A)$.
\end{corollary}

\begin{remark}
\label{re:openinsteadofsurjective}
For $A = M_2(k) \oplus M_2(k)$, $B = k \oplus M_2(k)$, the module $M = k^2 \oplus M_2(k)$ is invertible, hence implementing by \cref{PropositionInvertibleImplementing}, but the automorphism of $A$ that exchanges the two copies of $M_2(k)$ (which is not contained in $\Aut(A)_0$) is not implemented on $M$.
 Hence \cref{PropositionInvertibleImplementing} would be false if we would strengthen \cref{DefinitionImplementing} by requiring the maps $p_r$ and $p_\ell$  be surjective instead of open, see \cref{RemarkIdentityComponentsImplementing}.
\end{remark}

Our third result concerns super algebras $A$ whose first graded Hochschild cohomology vanishes, $\mathrm{HH}^1(A)=0$. 
$\mathrm{HH}^1(A)$ can be defined as the quotient $\mathrm{Der}(A)/\mathrm{InnDer}(A)$ of even derivations on $A$ by inner derivations on $A$.
We recall that graded Hochschild cohomology is a Morita invariant of super algebras \cite{Kassel1986,Blaga2006}. 

\begin{proposition} \label{PropHochschildImplementing}
Let $A$ and $B$ be super algebras and let $M$ be an $A$-$B$-bimodule.
\begin{enumerate}[(a)]
\item
If $\mathrm{HH}^1(A) = 0$, then $M$ is left implementing.

\item
If $\mathrm{HH}^1(B) = 0$, then $M$ is right implementing.

\end{enumerate}
\end{proposition}

\begin{proof}
The Lie algebra $\mathrm{Der}(A)$ of derivations on $A$ is the Lie algebra of $\Aut(A)$ and the Lie algebra $\mathrm{InnDer}(A)$ is the Lie algebra of $\mathrm{Inn}(A) = A^\times/Z(A)^\times$, the Lie group of inner automorphisms.
The condition $\mathrm{HH}^1(A) = \mathrm{Der}(A)/\mathrm{InnDer}(A) = 0$ implies that the quotient $\Aut(A)/\mathrm{Inn}(A)$ is discrete, hence its identity component is trivial. 
Therefore $p_\ell$ is trivially surjective. 
Similarly, $\mathrm{HH}^1(B) = 0$ implies that $\Aut(B)_0 = \{1\}$ and $p_r$ is surjective.
\end{proof}

The graded Hochschild cohomology of a \emph{separable} super algebra vanishes.
In the ungraded case, this is a result of Hochschild \cite[Thm. 4.1]{Hochschild1945}. In the graded case, we can see this as follows. As we are working with finite-dimensional algebras over $k=\R$ or $\C$, which are perfect fields, \emph{separable} super algebras are the same as \emph{semisimple} super algebras.
Here, a super algebra is called \emph{semisimple} if it is a direct sum of simple super algebras. As Hochschild cohomology is additive, it suffices to look at these simple factors separately.
By the graded Artin-Wedderburn theorem (see \cite{Jozefiak1988}), every simple super algebra $A$ is of the form $M_n(D)$, where $D$ is a super division algebra (\ie each homogeneous element of $D_i$ is invertible). These are Morita equivalent to Clifford algebras, which  have $\mathrm{HH}^n=0$ for $n\geq 1$ \cite[Prop. 1]{Kassel1986}.

Let us summarize these observations as follows.

\begin{corollary}
\label{co:ssimp}
\label{RemarkHHSemisimple}
If $A$ and $B$ are semisimple super algebras, then every $A$-$B$-bimodule is implementing.
\end{corollary}

\begin{remark}
The reason to formulate \cref{PropHochschildImplementing} using Hochschild cohomology (instead of just requesting all derivations of $A$ to be inner) is that this condition is more apparently Morita invariant, by Morita invariance of Hochschild cohomology.
\end{remark}

\begin{example}
A simple example of a super algebra $A$ where $\mathrm{HH}^1(A) \neq 0$ is the supercommutative super algebra $\Lambda k$.
It has no inner derivations since it is supercommutative, but $\mathrm{HH}^1(A) \cong k$, spanned by the even derivation $D$ determined by $D(\epsilon) = \epsilon$, where $\epsilon$ is the odd generator of $\Lambda k$.
\end{example}

\subsection{The relative tensor product of implementing bimodules}

By the following result we make sure that one can define a bicategory whose morphisms are implementing bimodules.

\begin{proposition} \label{PropositionTensorProdImplementing}
Let $A$, $B$, and $C$ be super algebras, let $M$ be an $A$-$B$-bimodule and let $N$ be a $B$-$C$-bimodule.
\begin{enumerate}[(1)]
\item \label{PropositionTensorProdImplementingA}
If $M$ is implementing and $N$ is left implementing, then $M \otimes_B N$ is left implementing.
\item \label{PropositionTensorProdImplementingB}
If $M$ is right implementing and $N$ is implementing, then $M \otimes_B N$ is right implementing.
\item \label{PropositionTensorProdImplementingC}
 If $M$ and $N$ are implementing, then $M \otimes_B N$ is implementing.
\end{enumerate}
\end{proposition}

\begin{proof}
We will show (1). Statement (2) follows analogously and statement (3) follows from combining (1) and (2).
By \cref{RemarkIdentityComponentsImplementing}, we have to show that the homomorphism $p_\ell :I(M \otimes_B N)_0 \to \Aut(A)_0$ is surjective.
\\
To this end, consider the fiber product $I(M) \times_{\Aut(B)} I(N)$ (over the map $p_r$ for $I(M)$ and $p_\ell$ for $I(N)$), which is  again a Lie group because the assumption that $M$ and $N$ are implementing implies that $p_\ell$ and $p_r$ are submersions (see \cref{RemarkIdentityComponentsImplementing}).
It can be identified with the set of tuples $(\alpha, u, \beta, v, \gamma)$, where $(\alpha, u, \beta) \in I(M)$ and $(\beta, v, \gamma) \in I(N)$.
One now checks that the map  
\begin{equation*}
I(M) \times_{\Aut(B)} I(N) \to I(M \otimes_B N)\quomma
(\alpha, u, \beta, v, \gamma) \mapsto (\alpha, u \otimes v, \gamma)
\end{equation*}
is well defined. 
\,
\\
Passing to identity components, consider the diagram
\begin{align} 
\label{DiagramInheritImplementing}
\xymatrix@C=1em@R=4em{
  & & (I(M) \times_{\Aut(B)} I(N))_0 \ar[d] \ar[dl] \ar[dr]& & \\
  & I(M)_0 \ar[dl] \ar[dr]& I(M \otimes_B N)_0 \ar[dll] \ar[drr]& I(N)_0 \ar[dl] \ar[dr]& \\
  \Aut(A)_0 & & \Aut(B)_0 & & \Aut(C)_0
}
\end{align}
of Lie groups.
The diamond in the middle commutes as it is the induced diagram of identity components associated to the pullback diagram defining $I(M) \times_{\Aut(B)} I(N)$,  while the commutativity of the two quadrangular diagrams  is easy to check.
Since $M$ and $N$ are implementing, all maps in the bottom row are surjective (see \cref{RemarkIdentityComponentsImplementing}).
\\
We claim that then also the projection $(I(M) \times_{\Aut(B)} I(N))_0 \to I(M)_0$ is surjective.
To show this, let $(\alpha, u, \beta) \in I(M)_0$, hence $\beta \in \Aut(B)_0$.
Since $N$ is implementing, there exists $(\beta, v, \gamma) \in I(N)_0$.
However, it is not clear that the element $(\alpha, u, \beta, v, \gamma) \in I(M)_0 \times_{\Aut(B)_0} I(N)_0$ is contained in the (generally smaller) set $(I(M) \times_{\Aut(B)} I(N))_0 \to I(M)_0$, so we have to be more precise in the choice of lift.
To this end, let $(\alpha_s, u_s, \beta_s)_{s \in [0, 1]}$, be a continuous path in $I(M)_0$ with $(\alpha_0, u_0, \beta_0) = (e, e, e)$ and $(\alpha_1, u_1, \beta_1) = (\alpha, u, \beta)$, the identity element of $I(M)$.
Then $(\beta_s)_{s \in [0, 1]}$ is a continuous path in $\Aut(B)_0$ connecting $\beta$ to the identity.
Since $N$ is implementing, $p_\ell : I(N)_0 \to \Aut(B)_0$ is a surjective homomorphism of Lie groups, hence a fiber bundle (see e.g., \cite[Thm.~E.3]{HallLieGroups}). 
Therefore, we can choose a lift $(\beta_s, v_s, \gamma_s)_{s \in [0, 1]}$ of the path $(\beta_s)_{s \in [0, 1]}$ to $I(N)_0$ starting at the identity.
Then by construction, $(\alpha, u, \beta, v_1, \gamma_1) \in (I(M) \times_{\Aut(B)} I(N))_0$ is a lift of $(\alpha, u, \beta) \in I(M)_0$.
This finishes the proof of the claim.
\\
Now, the composite map $(I(M) \times_{\Aut(B)} I(N))_0 \to I(M)_0 \to \Aut(A)_0$ is the composition of two surjective maps, hence surjective, and by commutativity of the diagram \eqref{DiagramInheritImplementing}, it factors through $p_\ell : I(M\otimes_B N)_0 \to \Aut(A)_0$.
Hence $p_\ell$ must be surjective.
\end{proof}

Due to \cref{PropositionTensorProdImplementing} we are  in position to define a sub-bicategory 
\begin{equation*}
\sAlgbiimp_k\subset \sAlgbi_k
\end{equation*}
 with the same objects (super algebras) but only the \textit{implementing} bimodules as 1-morphisms (and all intertwiners between those as 2-morphisms). We call $\sAlgbiimp_k$ the bicategory of \emph{implemented super algebras}. By \cref{PropositionInvertibleImplementing}, both bicategories have the same set of isomorphism classes of objects. In other words,  Morita equivalence is the same as isomorphism in any of these bicategories. One of the fundamental  insights of this paper is that in the context of  \textit{bundles}, it is the smaller bicategory $\sAlgbiimp_k$ which is relevant.

\subsection{Framing and symmetric monoidal structures}

\label{sec:smsonimplementedalgebras}

 We note that the framing 
$\sAlg_k \to \sAlgbi_k$
of \cref{lem:framingAlg} does not co-restrict to the sub-bicategory $\sAlgbiimp_k$, because the bimodule $B_{\varphi}$ is not necessarily implementing for all algebra homomorphisms $\varphi:A \to B$, see \cref{ExampleTwistedModuleNotImplementing}. Likewise,  the symmetric monoidal structure on $\sAlgbi_k$ does not restrict to the sub-bicategory $\sAlgbiimp_k$,  as the exterior tensor product of implementing bimodules is not necessarily implementing. For example, the bimodule from \cref{ExampleImplementing3} is a non-implementing module which is the exterior product of two implementing bimodules, compare \cref{ExampleImplementing2}.

Both problems can be solved simultaneously in two different ways by looking at two smaller sub-{}bi\-categories. \begin{enumerate}[(1)]
\item 
We restrict to the underlying sub-bigroupoid $\sAlgbigrpd k \subset \sAlgbi_k$ with only the invertible bimodules and invertible intertwiners. By \cref{PropositionInvertibleImplementing}, we have
\begin{equation*}
\sAlgbigrpd k \subset \sAlgbiimp_k\text{.}
\end{equation*}
Since the exterior tensor product of invertible bimodules is again invertible,  the symmetric monoidal structure restricts from $\sAlgbi_k$ to $\sAlgbigrpd{k}$.

For the framing, we may then restrict to the groupoid $\sAlggrpd k$ of super algebras and super algebra \textit{iso}morphisms. Then, by \cref{ExampleImplementing1}, we obtain a new framing
\begin{equation}
\label{eq:framingimp}
\sAlggrpd k \to \sAlgbigrpd k\text{,} 
\end{equation} 
which is proved like in \cref{lem:framingAlg}, combined with the fact that $\phi$ is invertible.

\item
We restrict to the full sub-bicategory 
$\sssAlgbi_k \subset \sAlgbi_k$ over all semisimple super algebras. By \cref{co:ssimp} we have 
\begin{equation*}
\sssAlgbi_k \subset \sAlgbiimp_k\text{.}
\end{equation*}

For the framing, it is clear again by \cref{co:ssimp} that the framing restricts to a framing
\begin{equation*}
\sssAlg_k \to \sssAlgbi_k\text{.}
\end{equation*}

Moreover, as the tensor product of semisimple super algebras is again semisimple, the symmetric monoidal structure restricts from $\sAlgbi_k$ to $\sssAlgbi_k$. 

\end{enumerate}

\begin{proposition}
The following table describes the dualizable, fully dualizable, and invertible objects in all three symmetric monoidal bicategories of super algebras: 
\begin{center}
\begin{tabular}{l|ccc}
 & dualizable & fully dualizable & invertible \\\hline
$\sAlgbi_k$ & all & semisimple & central simple \\
$\sAlgbigrpd k$ & central simple & central simple & central simple \\
$\sssAlgbi_k$ & all & all & central simple \\
\end{tabular}
\end{center}
\end{proposition}

\begin{proof}
The first line only repeats \cref{re:dualizable,TheoremCentralSimpleAlgebras}.
In the second line, the restriction to invertible bimodules requires $A \otimes A^{\opp}$ to be Morita equivalent to $k$ in all three cases; hence $A$ must already be invertible and thus central simple. The third line follows from the first.  
\end{proof}

\begin{remark}
All considerations of \cref{SectionImplementingModules} have ungraded counterparts: a sub-bicategory 
\begin{equation*}
\Algbiimp_k\subset \Algbi_k
\end{equation*}
of \emph{implemented algebras}, and two possible options for framed and symmetric monoidal versions: the sub-bicategories $\Algbigrpd k$ and $\ssAlgbi_k$.
\end{remark}

\section{The bicategory of  algebra bundles}       \label{sec:algebrabundles}     

In this section, we introduce a bundle version of the bicategory $\sAlgbiimp_k$ of  algebras defined in \cref{SectionImplementingModules}. We also discuss in detail why a similar bundle version of the bicategory $\sAlgbi_k$ of algebra bundles of \cref{sec:bicatalgebras} does not exist. 

\subsection{Algebra bundles and bimodule bundles}

Let $X$ be a smooth manifold. 

\begin{definition}[Super algebra bundle]
\label{def:algebrabundle}
A \emph{super algebra bundle over} $X$ is a smooth vector bundle $\pi:\mathcal{A} \to X$, with the structure of a super algebra on each fibre $\mathcal{A}_x$, $x \in X$, such that each point in $X$ has an open neighborhood $U \subset X$ for which there exists a super algebra $A$ and a diffeomorphism $\phi: \mathcal{A}|_U \to U \times A$ that preserves fibres and restricts to a super algebra isomorphism $\phi_x: \mathcal{A}_x \to A$ in each fibre over $x\in U$. 
A homomorphism between two super algebra bundles is a vector bundle map that is a grading-preserving algebra homomorphism in each fiber.
Super algebra bundles over $X$ and homomorphisms form a category $\sAlgBdl_k(X)$. \end{definition}

\begin{remark}
Super algebra bundles are \emph{not} the same as monoid objects in the category of super vector bundles (the same statement holds in the non-super case).
The problem is that a monoid object may fail to be locally trivializable as an \emph{algebra} bundle.
An example of a monoid object which is not locally trivializable can be found, \eg in \cite[Example 1.2.4]{Mertsch2020}.
\end{remark}
 
\begin{remark}
\label{re:typicalfibre}
In \cref{def:algebrabundle}, we allow vector bundles of varying rank over different connected components. Thus, super algebra bundles  do not necessarily have a single typical fibre, in the sense that the super algebras $A$ of all local trivializations could be chosen to be the same. However, a straightforward argument shows that the restriction of a super algebra bundle to a connected component of $X$ does have a single typical fibre. 
\end{remark} 

\begin{definition}[Bimodule bundle]
\label{def:bimodulebundle}
Let $\mathcal{A}$ and $\mathcal{B}$ be super algebra bundles over $X$ and let $\mathcal{M}$ be a super vector bundle over $X$ with the structure of an $\mathcal{A}_x$-$\mathcal{B}_x$-bimodule  in each fiber $\mathcal{M}_x$.
\begin{enumerate}[(i)]
\item
        A \emph{local trivialization} for $\mathcal{M}$ is an open set $U \subset X$ together with a tuple $(A, M, B; \phi, u, \psi)$ of super algebras $A$ and $B$, an $A$-$B$-bimodule $M$ and local trivializations $\phi: \mathcal{A}|_U \to U \times A$, $\psi:\mathcal{B}|_U \to U \times B$ (as super algebra bundles) and $u: \mathcal{M}|_U \to U \times M$ (as a vector bundle) such that fiberwise, $u$ is an intertwiner along $\phi$ and $\psi$.

\item 
$\mathcal{M}$ is an $\mathcal{A}$-$\mathcal{B}$-\emph{bimodule bundle} if each point $x \in X$ has an open neighborhood $U$ over which there exists a local trivialization for $\mathcal{M}$.
\end{enumerate}
Morphisms between $\mathcal{A}$-$\mathcal{B}$-bimodule bundles are super vector bundle morphisms that are even intertwiners in each fibre (they will again be called intertwiners).
$\mathcal{A}$-$\mathcal{B}$-bimodule bundles and their morphisms form a category we denote by $\sBimodBdl_{\mathcal{A},\mathcal{B}}(X)$. 
\end{definition}

\begin{example} \label{ExampleTwistedModuleIso}
Let $\varphi:\mathcal{A} \to \mathcal{B}$ be an isomorphism of super algebra bundles over $X$. Then,  there is a $\mathcal{B}$-$\mathcal{A}$-bimodule bundle $\mathcal{B}_\varphi$, which over a point $x \in X$ has fibers $(\mathcal{B}_x)_{\varphi_x}$.
A local trivialization of $\mathcal{B}_{\varphi}$ is given by $(B, B_{\varphi_0}, A; \psi, \psi, \varphi_0^{-1}\circ \psi \circ \varphi)$, where $B$ and $A$ are typical fibres of $\mathcal{B}$ and $\mathcal{A}$, respectively,  $\psi$ is a local trivialization of $\mathcal{B}$, and $\varphi_0: A \to B$ is an arbitrary fixed super algebra isomorphism. 
\end{example}

The example above does \emph{not} generalize to non-invertible bundle homomorphisms $\varphi:  \mathcal{A} \to \mathcal{B}$.
For a detailed discussion of this, see \cref{ex:loctrivfail1,ex:loctrivfail2} in  \cref{SectionFraming}.

\begin{remark} \label{RemarkNoTypicalFiber}
An alternative, non-equivalent  notion of an $\mathcal{A}$-$\mathcal{B}$-bimodule bundle is to define it as a super vector bundle $\mathcal{M}$ together with a morphism of algebra bundles $\mathcal{A}\otimes\mathcal{B}^{\opp} \to \End(\mathcal{M})$.
This notion is strictly more general than our \cref{def:bimodulebundle}.
Indeed, the collection $({A}_{\varphi(x)})_{x \in X}$ of $A$-$B$-bimodules, for a smooth map $\varphi: X \to \Hom(B, A)$, is certainly an $\underline{A}$-$\underline{B}$-bimodule bundle in this alternative sense, but -- as just said --  it is not locally trivial in the sense of our \cref{def:bimodulebundle}.
\\
We show, however, in \cref{prop:weakernotioncoincides} that this alternative notion of bimodule bundles is equivalent to \cref{def:bimodulebundle} when the super algebra bundles $\mathcal{A}$ and $\mathcal{B}$ have semisimple fibres.
\\
Anyway, this alternative notion of bimodule bundles does not, in general, admit relative tensor products, as can be seen from \cref{ExNoTypicalFiber}; thus it cannot be used for our purposes.
\end{remark}

The transition functions between a pair of local trivializations $(A, M, B; \phi, u, \psi)$ and $(A^\prime, M^\prime, B^\prime; \phi^\prime, u^\prime, \psi^\prime)$, defined over $U$ and $U^\prime$, are  smooth maps  
\begin{equation*}
U \cap U^\prime \to \mathrm{Hom}_{\Impcat_k}((A, M, B), (A^\prime, M^\prime, B^\prime))\text{.}
\end{equation*}
Even in the case that $A = A^\prime$, $B = B^\prime$, $M = M^\prime$ they are in general not bimodule automorphisms of $M$; instead, they lie in the larger group of implementers (see \cref{re:enlargement}).
In particular, the bimodule $M$ appearing in local trivializations is not unique up to bimodule automorphisms, but rather unique up to isomorphism in the category $\Impcat_k$. For example, one can always replace $M$ by $\pss{\phi} M_\psi$, for any automorphisms $\phi$ and $\psi$ (\cref{re:twistedbimoduleimplementers}), whereas $M \cong \pss{\phi} M_\psi$ as $A$-$B$-bimodules only if $\phi$ and $\psi$ are inner (\cref{lem:framing:d}).
With these considerations in mind, we come to the following definition.

\begin{definition}[Typical fiber of module bundles]
Let $\mathcal{A}$, $\mathcal{B}$ be super algebra bundles over $X$  and let $\mathcal{M}$ be an $\mathcal{A}$-$\mathcal{B}$-bimodule bundle.
A \emph{typical fibre} for $\mathcal{M}$ is an object  $(A,M,B)$ in $\Impcat_k$, such  that around every point of $X$ there exist local trivializations of the form $(A, M, B;\phi,u,\psi)$. 
\end{definition}

Similar to \cref{re:typicalfibre}, one can show that every bimodule bundle has a typical fiber over each connected component of $X$, and, as said before, that typical fibre is unique up to isomorphism in $\Impcat_k$.

\begin{remark}
\label{rem:fibreimplementing}
By definition of a local trivialization, \cref{def:bimodulebundle}, $(A, M, B)$ is a typical fiber of an $\mathcal{A}$-$\mathcal{B}$-bimodule bundle if and only if $(\mathcal{A}_x,\mathcal{M}_x,\mathcal{B}_x) \cong (A,M,B)$ in $\Impcat_k$ for each $x \in X$.
\end{remark}

\subsection{Implementing bimodule bundles}

\label{sec:implementingbimodulebundles}

The following is one of the central definitions of this article.

\begin{definition}[Implementing bimodule bundle]
\label{def:implementingbimodulebundle}
Let $\mathcal{A}$ and $\mathcal{B}$ be super algebra bundles over $X$ and let $\mathcal{M}$. An $\mathcal{A}$-$\mathcal{B}$-bimodule bundle $\mathcal{M}$ is called \textit{(left/right) implementing} if all
fibres $\mathcal{M}_x$ are (left/right) implementing. 
\end{definition}

\begin{remark}
By \cref{RemarkTwistedModulesImplementing,rem:fibreimplementing}, a bimodule bundle over $X$ is (left/right) implementing if and only if its typical fiber over each connected component of $X$ is (left/right) implementing.
\end{remark}

Let $\mathcal{A}$, $\mathcal{B}$, $\mathcal{C}$ be super algebra bundles, let $\mathcal{M}$ be an $\mathcal{A}$-$\mathcal{B}$-bimodule bundle and $\mathcal{N}$ a $\mathcal{B}$-$\mathcal{C}$-bimodule bundle. 
Let $\mathcal{M} \otimes_{\mathcal{B}} \mathcal{N}$ be the fiberwise relative tensor product, which is in the first place a collection of bimodules $\mathcal{M}_x \otimes_{\mathcal{B}_x} \mathcal{N}_x$, indexed by $x \in X$.
As announced, it is not always possible to endow $\cup_{x \in X}\mc{M}_{x} \otimes_{\mc{B}_{x}} \mc{N}_{x} = \mathcal{M} \otimes_{\mathcal{B}} \mathcal{N}$ with the structure of a vector bundle (see \cref{ExNoTypicalFiber} below), but we have the following positive result.

\begin{proposition} \label{PropositionDefinitionTensorProduct}
Suppose that either $\mathcal{M}$ is right implementing or that $\mathcal{N}$ is left implementing.
Then, $\mc{M} \otimes_{\mc{B}} \mc{N}$ has a unique structure of vector bundle over $X$, such that the canonical map $\mc{M} \otimes \mc{N} \rightarrow \mc{M} \otimes_{\mc{B}} \mc{N}$ is smooth.
Moreover, when each of the fibers $\mc{M}_{x} \otimes_{\mc{B}_{x}} \mc{N}_{x}$ is equipped with the obvious $\mc{A}_{x}$-$\mc{C}_{x}$-bimodule structure, $\mc{M} \otimes_{\mc{B}} \mc{N}$ is an $\mc{A}$-$\mc{C}$-bimodule bundle.
Finally, if $\mathcal{M}$ has typical fibre $(A,M,B)$ and $\mathcal{N}$ has typical fibre $(B,N,C)$, then there exists an automorphism $\phi$ of $B$ such that $\mathcal{M} \otimes_{\mathcal{B}} \mathcal{N}$ has typical fibre $(A,M \otimes_B \pss{\phi}N,C)$.
\end{proposition}

\begin{proof}
Let $U \subset X$ be an open set and let $(A, M, B; \phi, u, \psi)$ and $(B, N, C; \tilde{\psi}, v, \tau)$ be  local trivializations for $\mathcal{M}$, respectively $\mathcal{N}$, defined over $U$.
The point here is that we may assume that the typical fiber $B$ of $\mathcal{B}$ is the same in both local trivializations, but we can \emph{not} generally assume that $\psi = \tilde{\psi}$.
We want to construct a local trivialization of $\mathcal{M} \otimes_{\mathcal{B}} \mathcal{N}$.
Suppose that $\mathcal{N}$ is left implementing and set $\varphi = \tilde{\psi} \circ \psi^{-1}$, which we view as a smooth function $U \to \Aut(B)$.
We claim that for each $x \in U$, we get an intertwiner 
\begin{equation*}
u_x \otimes v_x : \mathcal{M}_x \otimes_{\mathcal{B}_x} \mathcal{N}_x \to M \otimes_B \pss{\varphi_x} N
\end{equation*}
along $\phi$ and $\tau$.
To show that this is well defined, we have to show that elements of the form $(m \ract b) \otimes n - m \otimes (b \ract n)$, $m \in \mathcal{M}_x$, $n \in \mathcal{N}_x$, $b \in \mathcal{B}_x$, are sent to zero.
Here we calculate
\begin{align*}
  (u_x \otimes v_x)((m \ract b) \otimes n - m \otimes (b \ract n)) 
  &= u_x(m \ract b) \otimes v_x(n) - u_x(m) \otimes v_x(b \ract n) \\
  &= (u_x(m) \ract \psi_x(b)) \otimes v_x(n) - u_x(m) \otimes (\tilde{\psi}_x(b) \lact v_x(n) )\\
  &= u_x(m) \otimes (\psi_x(b) \lact_{\varphi_x} v_x(n) - u_x(m) \otimes (\tilde{\psi}_x(b) \lact v_x(n) )\\
  &= u_x(m) \otimes ((\varphi_x(\psi_x(b))  - \tilde{\psi}_x(b)) \lact v_x(n) ),
\end{align*}
which is zero by the definition of $\varphi$.
\\
Let $\varphi_0 \in \Aut(B)$ such that $\tilde{\varphi} := \varphi_0^{-1} \circ \varphi$ takes values in $\Aut(B)_0$ (where $\varphi_0$ is identified with a constant function $U \to \Aut(B)$).
Since $N$ is left implementing,  $I(N)_0 \to \Aut(B)_0$ is a surjective submersion.
Therefore (after possibly shrinking $U$), $\tilde{\varphi}$ has a lift $\hat{\varphi}: U \to I(N)_0$ along $p_\ell$.
This means that $\hat{\varphi}_x$ is a left implementer for $\tilde{\varphi}_x$, \ie $\hat{\varphi}_x(a \lact n) = \tilde{\varphi}_{x}(a) \lact \hat{\varphi}_x(n)$.
Set $\sigma := p_r \circ \hat{\varphi}: U \to \Aut(C)_0$.
We now let $w$ be the local trivialization of $\mathcal{M} \otimes_{\mathcal{B}} \mathcal{N}$ over $U$ defined fibrewise by
\begin{equation} \label{LeftTrivialization}
\xymatrix@C=1.2cm{
w_x : \mathcal{M}_x \otimes_{\mathcal{B}_x} \mathcal{N}_x \ar[r]^-{u_x \otimes v_x} &  M \otimes_B \pss{\varphi_x} N \ar[r]^{1 \otimes \hat{\varphi}_x^{-1}} & M \otimes_B \pss{\varphi_0} N_{\sigma_x^{-1}}.
}
\end{equation}
It is an intertwiner along $\phi$ and $\sigma^{-1} \circ \tau$.
\\
Let $w^\prime$ be another trivialization, constructed following the procedure above, over an open subset $V \subseteq X$.
The transition function
\begin{equation*}
w^\prime w^{-1} : (U \cap V) \times M \otimes_B \pss{\varphi_0} N \to M^\prime \otimes_B \pss{\varphi_0^\prime} M^\prime
\end{equation*}
is fiberwise a vector space isomorphism,
which is an implementer along $\phi^\prime \circ \phi^{-1}$ and $(\sigma^\prime)^{-1} \circ \tau^\prime \circ \tau^{-1} \circ \sigma$.
To see that it is smooth, observe that the fibers are
\begin{equation*}
  w_x^\prime w_x^{-1} = (1 \otimes (\hat{\varphi}^\prime)_x^{-1})(u^\prime_x \otimes v^\prime_x)(u_x \otimes v_x)^{-1}(1 \otimes \hat{\varphi}_x^{-1})^{-1}
  = u^\prime_x u_x^{-1} \otimes (\hat{\varphi}^\prime)_x^{-1} v_x^\prime v_x^{-1} \hat{\varphi}_x.
\end{equation*}
Since $u$, $u^\prime$ and $v$, $v^\prime$ are trivializations of $M$, respectively $N$, the transition functions $u^\prime u^{-1}$ and $v^\prime v^{-1}$ are smooth.
Since also $\hat{\varphi}^\prime$ and $\hat{\varphi}$ are smooth, the transition function $w^\prime w^{-1}$ is smooth as well.
\\
We have now constructed local trivializations of $\mathcal{M} \otimes_{\mathcal{B}} \mathcal{N}$ over a neighborhood of every point in $X$ and shown that transition functions for any two of these local trivializations are smooth.
This gives $\mathcal{M} \otimes_{\mathcal{B}}\mathcal{N}$ the structure of a vector bundle.
One easily checks that the map $\mathcal{M} \otimes \mathcal{N} \to \mathcal{M} \otimes_{\mathcal{B}}\mathcal{N}$ is smooth and surjective.
It is a submersion as it is fiberwise linear.
Uniqueness follows from the well-known fact that for a surjective map $\pi: M \to S$ from a manifold $M$ to a set $S$, there exists at most one manifold structure on $S$ turning $\pi$ into a submersion.
This finishes the proof in the case that $\mathcal{N}$ is left implementing.
\\
If $\mathcal{M}$ is right implementing, we consider local trivializations $w$  defined fibrewise by
\begin{equation} \label{RightTrivialization}
\xymatrix@C=1.2cm{
w_x : \mathcal{M}_x \otimes_{\mathcal{B}_x} \mathcal{N}_x 
\ar[r]^-{u_x \otimes v_x}
 &  M_{\varphi_x^{-1}} \otimes_B N 
 \ar[r]^-{{\hat{\varphi}}_x \otimes 1} 
 & \pss{\sigma} M_{\varphi_0^{-1}} \otimes_B N,
}
\end{equation}
where now $\hat{\varphi} : U \to I(M)_0$ is a lift of $\tilde{\varphi}$ along $p_r$ and $\sigma = p_\ell \circ \hat{\varphi}$.
The proof that transition functions for this set of trivializations is smooth is analogous to the previous case.
\\
Observe finally that if $\mathcal{M}$ is right implementing and $\mathcal{N}$ is left implementing, then we can consider trivializations of both type \eqref{LeftTrivialization} and type \eqref{RightTrivialization}.
    These are compatible, as $M \otimes_B \pss{\varphi_0} N_{\sigma_x^{-1}} \cong \pss{\sigma^\prime_x} M_{\varphi_0^{-1}} \otimes_B N$ via $\hat{\varphi}^\prime_x \otimes \hat{\varphi}_x$.
\end{proof}

Without the assumption that one of $\mathcal{M}$ and $\mathcal{N}$ is implementing, \cref{PropositionDefinitionTensorProduct} is false, as the following example shows.

\begin{example} \label{ExNoTypicalFiber}
Let $\mathcal{A}=\underline{A}$ be the  trivial super algebra bundle  over $X = \mathrm{GL}_2(k)$, with 
\begin{equation*}
A = \Lambda k^2 = \Lambda k \otimes \Lambda k\text{,}
\end{equation*}
the exterior algebra on two generators $\epsilon_1$ and $\epsilon_2$.
Let 
\begin{equation*}
M = \Lambda k \otimes \Lambda^1 k  = \{ \lambda \epsilon_1 + \mu \epsilon_1 \epsilon_2\mid \lambda, \mu \in k\},
\end{equation*}
 an $A$-$A$-bimodule.
Consider the bundle $\mathcal{M} = \underline{M}_\varphi$, where  $\varphi: X \to \Aut(A)$ is such that for $x \in X = \mathrm{GL}_2(k)$, $\varphi_x$ is the automorphism of $\Lambda k^2$ induced by the linear transformation $x$ of $k^2$.
In other words, at $x \in X$, $\mathcal{A}_x=A$ acts on $\mathcal{M}_x$ from the left via the standard action and from the right via $\varphi$.
A global trivialization of $\underline{M}_\varphi$ is $(A, M, A; \mathrm{id}, \mathrm{id}, \varphi)$. 
\\
If now $\mathcal{N} = \underline{M}$ is the trivial bimodule bundle, then the fiberwise tensor product $\mathcal{M} \otimes_{\mathcal{A}} \mathcal{N}$ does not have the structure of a vector bundle, as it has varying fiber dimensions.
Indeed, over $x = \mathrm{id} \in \mathrm{GL}_2(k)$, we have $\dim_k(\mathcal{M}_x \otimes_{\mathcal{A}_x} \mathcal{N}_x) = 2$ (the fiber is isomorphic to $M$ as a module), while when $x$ is the flip $\epsilon_1 \to - \epsilon_2$, $\epsilon_2 \to \epsilon_1$, in $M_{\phi_x} \otimes_A M$, we have
\begin{align*}
  \epsilon_1 \epsilon_2 \otimes m &= - (\epsilon_1 \ract_{\varphi_x} \epsilon_1) \otimes m = -\epsilon_1 \otimes (\epsilon_1 \lact m) = 0, \\
  m \otimes \epsilon_1 \epsilon_2 &= - m \otimes (\epsilon_2 \lact \epsilon_1) = - (m \ract_{\varphi_x} \epsilon_2) \otimes \epsilon_1 = (m \ract \epsilon_1) \otimes \epsilon_1 = 0,
\end{align*}
for any $m \in M$, as $\epsilon_1$ acts trivially on $M$.
Therefore $\dim_k(\mathcal{M}_x \otimes_{\mathcal{A}_x} \mathcal{N}_x) = 1$. 
\\
Indeed, we recall that by \cref{ExampleImplementing3}, $\mathcal{M}$ is neither left nor right implementing.
\end{example}

We also would like to emphasize that in \cref{PropositionDefinitionTensorProduct} the  tensor product bundle $\mathcal{M} \otimes_{\mathcal{B}} \mathcal{N}$ does \emph{not} necessarily have typical fiber $(A, M \otimes_B N, C)$, but only $(A, M \otimes_B \pss{\phi}N, C)$, where $\phi$ is some automorphism of $B$. That these are not isomorphic happens in the following example.

\begin{example}
An explicit example where $M \otimes_B N$ and $M \otimes_B \pss{\phi} N$ have different dimension as $k$-vector spaces, hence $(A,M \otimes_B N,C) \ncong (A, M \otimes_B \pss{\phi} N,C)$ in $\Impcat_k$, is the following.
Let $A = C = k \oplus M_2(k)$, $B = M_2(k) \oplus M_2(k)$.
Let moreover $N = k^2 \oplus M_2(k)$, an $A$-$B$-bimodule and $M = (k^2)^\vee \oplus M_2(k)$, a $B$-$C$-bimodule (here $(k^2)^\vee$ denotes the dual vector space of $k^2$, an $M_2(k)$-$k$-bimodule).
Then the tensor product $M \otimes_B N$ is $k \oplus M_2(k)$, the identity bimodule.
On the other hand, if $\phi$ is the ``flip'' automorphism of $B$ that exchanges the two summands, then $M\otimes_B \pss{\phi}N \cong ((k^2)^\vee \oplus k^2)_{{\sigma}}$, where $A$ acts from left the standard way, while $B$ acts from the right along the flip isomorphism ${\sigma} : k \oplus M_2(k) \to M_2(k) \oplus k$.
Observe that $\dim_k(M \otimes_B N) = 5$, while $\dim_k(M \otimes_B\pss{\phi}N) = 4$.
\end{example}
Together with \cref{PropositionTensorProdImplementingC}, we obtain from \cref{PropositionDefinitionTensorProduct} the following main result. 

\begin{theorem}
\label{prop:tensorproductimplementing}
The relative tensor product of implementing bimodule bundles is an implementing bimodule bundle.
\end{theorem}

For super algebra bundles $\mathcal{A}$, $\mathcal{B}$ over $X$, we denote by
\begin{equation*}
  \sBimodBdlimp_{\mathcal{A},\mathcal{B}}(X) \subset \sBimodBdl_{\mathcal{A},\mathcal{B}}(X)
\end{equation*}
the full subcategory consisting of implementing bimodule bundles.  
Then, the relative tensor product of  \cref{prop:tensorproductimplementing} induces a functor
\begin{equation} \label{Composition}
\sBimodBdlimp_{\mathcal{A},\mathcal{B}}(X) \times \sBimodBdlimp_{\mathcal{B},\mathcal{C}}(X) \to \sBimodBdlimp_{\mathcal{A},\mathcal{C}}(X)\text{.}
\end{equation}   
It is unproblematic to construct associators and unitors for this tensor product functor, turning it into the composition of 1-morphisms in a bicategory whose objects are super algebra bundles.

\begin{definition}[Bicategory of  algebra bundles] \label{DefinitionPrestwoVectBdl}
The bicategory $\sAlgBdlbi_k(X)$ of \emph{ super algebra bundles} over $X$ has objects super algebra bundles over $X$, 1-morphisms implementing bimodule bundles over $X,$ and 2-morphisms even intertwiners.
The composition is given by \eqref{Composition}.
Similarly, we denote by $\AlgBdlbi_k(X)$ the sub-bicategory consisting of ungraded (\ie purely even) algebra bundles and ungraded implementing bimodule bundles.
\end{definition}

We will shortly require the following result about  invertibility of 1-morphisms and 2-morphisms in the bicategory $\sAlgBdlbi_k(X)$.

\begin{lemma}
\label{lem:fibrewiseinvertibility}
Let $\mathcal{A},\mathcal{B}$ be super algebra bundles over $X$,  and let $\mathcal{M},\mathcal{M}'$ be implementing $\mathcal{A}$-$\mathcal{B}$-bimodule bundles over $X$.
\begin{enumerate}[ (a)]
\item
\label{lem:fibrewiseinvertibility:a}
An intertwiner $\mathcal{M} \to \mathcal{M}'$ is invertible if and only if it is fibrewise invertible.

\item
\label{lem:fibrewiseinvertibility:b}
 $\mathcal{M}$ has a left (right) adjoint if and only if  each fibre has a left (right) adjoint.

\item 
\label{lem:fibrewiseinvertibility:c}
 $\mathcal{M}$ is invertible if and only if it is fibrewise invertible.

\end{enumerate}
\end{lemma}

\begin{proof}
(a) is clear and only stated for completeness.  
\\
In (b) and (c), only the \quot{if}-parts are non-trivial. 
We may assume that $X$ is connected, otherwise we consider each connected component separately. 
Let $(A, M, B)$ be a typical fibre of $\mathcal{M}$.
We will use that if an $A$-$B$-bimodule $M$ admits a left adjoint $L$, then a concrete model for such an adjoint is the dual $B$-$A$-bimodule $L = \underline{\Hom}_{k\text{-}B}(M, B)$, with evaluation $\varepsilon: L \otimes_A M \to B$ given by $\xi \otimes x \mapsto \xi(x)$.
If $M$ is implementing, then so is $L$, by \cref{PropAdjointImplementing}.
\\
Let $\mathcal{L}$ be the bundle $\underline{\mathrm{Hom}}_{k\text{-}\mathcal{B}}(\mathcal{M},\mathcal{B})$, with fibrewise $\mathcal{B}$-$\mathcal{A}$-bimodule fibrewise bimodule structure given by $(b \lact \xi \ract a)(x) := b\lact \xi(a \lact x)$. 
It follows from the remarks above that if $\mathcal{M}$ is implementing, then so is $\mathcal{L}$.
Let $\tilde{\varepsilon}:  \mathcal{L}  \otimes_{\mathcal{A}} \mathcal{M}\to \mathcal{A}$ be the intertwiner given by $\tilde{\varepsilon}(\xi \otimes x) = \xi(x)$.
Now by the algebraic statement above, if the fibres of $\mathcal{M}$ admit left adjoints, then for each $x \in X$, there exist bimodule maps $\tilde{\eta}_x: \mathcal{A}_x \to \mathcal{M}_x \otimes_{\mathcal{B}_x} \mathcal{L}_x$ such that $\tilde{\varepsilon}_x$ and $\tilde{\eta}_x$ witness $\mathcal{L}_x$ as a left adjoint of $\mathcal{M}_x$.
\\
It remains to show that the fibrewise defined bimodules maps ${\eta}_x$ assemble to a smooth bundle morphism. 
A local trivialization $(A, M, B; \phi, u, \psi)$ of $\mathcal{M}$ over some open set $U \subset X$ induces a local trivialization of $\mathcal{L}$, identifying it fibrewise with $L = \underline{\Hom}_{k\text{-}B}(M, B)$.
These trivializations identify the bundle map $\tilde{\varepsilon}$ with the bimodule map $\varepsilon$ over each point in $U$ and the maps $\tilde{\eta}_x$ yield a family of bimodule maps $\eta_x: A \to M \otimes_B L$, $x \in U$. 
Since each $\tilde{\eta}_x$ is the coevaluation of the adjunction $(\mathcal{L}_x, \varepsilon_x, \eta_x)$, each $\eta_x$ is the coevaluation of the adjunction $(L, \varepsilon, \eta_x)$. 
However, by uniqueness of adjunctions, $\eta_x$ is determined by $L$ and $\varepsilon$, hence $\eta_x$ is independent of $x$, hence constant.
Transferring back via the local trivialization, this shows that $\tilde{\eta}_x$ depends smoothly on $x$.
This finishes the proof of (b). 
\\
For (c), observe that if $\mathcal{M}$ is invertible, then any left adjoint is an inverse.
\end{proof}

\begin{definition}[Morita class]
\label{def:preMoritaclass}
Let $\mathcal{A}$ be a {super algebra} bundle over $X$ and let $A$ be a super algebra. Then, we say that \emph{$\mathcal{A}$ is of Morita class $A$}, if local trivializations $\mathcal{A}|_U \cong U \times A_U$ can be chosen around every point in $X$ such that $A_U$ and $A$ are Morita equivalent.
\end{definition}

The following result shows that the Morita class is an appropriate invariant of super algebra bundles as objects of the bicategory $\sAlgBdlbi_k(X)$.

\begin{lemma}
\label{lem:preMoritaclass}
Let $\mathcal{A}$ be a super algebra bundle over $X$.
\begin{enumerate}[ (a)]
\item
\label{lem:preMoritaclass:a}
If $X$ is connected, then there exists a super algebra $A$ such that $\mathcal{A}$ is of Morita class $A$. In fact, $\mathcal{A}$ is of Morita class $\mathcal{A}_x$, where $\mathcal{A}_x$ is the fibre of $\mathcal{A}$ over any point $x\in X$.

\item 
\label{lem:preMoritaclass:b}
Let $A$ and $B$ be super algebras, and let $\mathcal{A}$ be of Morita class $A$. Then $\mathcal{A}$ is  of Morita class $B$ if and only if $A$ and $B$ are Morita equivalent. 

\item
\label{lem:preMoritaclass:c}
Let $\mathcal{B}$ be another super algebra bundle with $\mathcal{A}\cong \mathcal{B}$ in the bicategory $\sAlgBdlbi_k(X)$, and let $A$ be a super algebra. 
Then $\mathcal{B}$ is of Morita class $A$ if and only if $\mathcal{A}$ is of Morita class $A$.  
\end{enumerate}
\end{lemma}

\begin{proof}
(a) As remarked above, any super algebra bundle over a connected manifold even has a typical fibre, and thus, in particular, it has this typical fibre as its Morita class. (b) is trivial. In (c), if $\mathcal{B}$ is of Morita class $A$ and $\mathcal{M}$ is an invertible $\mathcal{A}$-$\mathcal{B}$-bimodule bundle over $X$, then each local trivialization of $\mathcal{M}$ exhibits the  typical fibre  $A_{U}$ of $\mathcal{A}$ as Morita equivalent to the typical fibre $B_U$ of $\mathcal{B}$. By assumption, $B_U$ is also Morita equivalent to $A$; this shows that $A_U$ is also Morita equivalent to $A$ and thus that $\mathcal{A}$ is of Morita class $A$.     
\end{proof}

\begin{remark}
One can regard the Morita class of a super algebra bundle as a generalization of the rank of a vector bundle. This will become important when we set up the theory of 2-vector bundles in \cite{Kristel2020}. 
\end{remark}

\subsection{Framings for  algebra bundles} \label{SectionFraming}

In \cref{sec:smsonimplementedalgebras} we considered two sub-bicategories of $\sAlgbiimp_{k}$ which admit a framing by (the corresponding sub-bicategory of) $\sAlg_{k}$.
In this section, we will explain how to transfer these results to $\sAlgBdlbi_k$.
The first option is to restrict to the bigroupoid
\begin{equation*}
\sAlgBdlbigrpd kX \subset \sAlgBdlbi_k(X)\text{.} 
\end{equation*}
Then, we obtain a functor \begin{equation} \label{FramingForInvertibles}
\sAlgBdlgrpd kX \to \sAlgBdlbigrpd kX
\end{equation}
for each manifold $X$.
Here, $\sAlgBdlgrpd kX$ is the subgroupoid of $\sAlgBdl_k(X)$ with only the super algebra bundle \textit{iso}morphisms. 
To see this, we note that for a super algebra bundle isomorphism $\varphi:\mathcal{A} \to \mathcal{B}$ we have a $\mathcal{B}$-$\mathcal{A}$-bimodule bundle $\mathcal{B}_{\varphi}$ (\cref{ExampleTwistedModuleIso}) and  this bimodule bundle is implementing (\cref{PropositionInvertibleImplementing}).
By \cref{lem:fibrewiseinvertibility:b}, we obtain the following result.

\begin{lemma} \label{PropFraming1}
The functor \eqref{FramingForInvertibles} is a framing for the bigroupoid of  $\sAlgBdlbigrpd kX$ of  super algebra bundles. 
\end{lemma}

In the remainder of this subsection we show that the second option pursued in \cref{sec:smsonimplementedalgebras} works also here, namely the restriction to semisimple super algebras.  For this purpose, we need to explore under which circumstances a bimodule bundle $\mathcal{B}_{\varphi}$ can be defined, when $\varphi$ is a not necessarily invertible homomorphism of super algebra bundles.   

We first discuss how $\mathcal{B}_\varphi$ may fail to be locally trivial for general homomorphisms $\varphi: \mathcal{A} \to \mathcal{B}$.
As this is a local question, we may assume that $\mathcal{A} = \underline{A}$, $\mathcal{B} = \underline{B}$ are trivial bundles, and $\varphi: X \to \Hom(A, B)$ is a smooth map, and then consider the super vector bundle $\underline{B}_{\varphi}$, which is equal to $\underline{B}$ as a super vector bundle, and equipped with the  $A$-$B$-bimodule structure of $B_{\varphi_x}$ over the point $x\in X$.

\begin{lemma} \label{LemmaLocallyTrivial}
        The bundle $\underline{B}_{\varphi}$ is a bimodule bundle in the sense of \cref{def:bimodulebundle} if and only if each point $x_0 \in X$ has an open neighborhood $U$ on which there exist functions $\phi: U \to \Aut(B)$ and $\psi: U \to \Aut(A)$ with 
  \begin{equation} \label{ConditionLocallyTrivial}
    \phi_x \circ \varphi_x = \varphi_{x_0} \circ \psi_{x}\text{.}
  \end{equation}
\end{lemma}

\begin{proof}
Suppose there exists a local trivialization $(B, M, A; \phi, u, \psi)$ over an open set $U\subset X$. We can assume that $M=B_{\varphi_0}$ for some fixed algebra homomorphism  $\varphi_0 : A \to B$.
For all $x \in U$, this data has to satisfy
\begin{equation*}
  u_x(bc \ract \varphi_x(a)) = \phi_x(b) \lact u_x(c) \ract (\varphi_0 \circ \psi)(a), \qquad a \in A, ~~ c \in B_{\varphi_0}~~ b \in B.
\end{equation*}
Since $u^\prime := \phi^{-1} \circ u$ commutes with left multiplication by elements of $B$, it must be of the form $u_x(b) = b c_x$ for a smooth function $c: U \to B$ (explicitly, $c_x = u_x(1)$).
We then have
\begin{equation*}
  \varphi_x(b) c_x = u^\prime(1 \ract \varphi_x(a)) = u^\prime(1) \ract (\phi_x^{-1} \circ \varphi_0 \circ \psi_x)(a) = c_x (\phi_x^{-1} \circ \varphi_0 \circ \psi_x)(a),\qquad a \in A,
\end{equation*}
in other words, we have $\phi_x^\prime \circ \varphi_x = \varphi_0 \circ \psi$, with $\phi^\prime = \phi \circ i(c)^{-1}$, where $i(c) \in \Aut(B)$ denotes pointwise conjugation by $c$.
\\
Conversely, given functions $\phi$ and $\psi$ as in the statement of the lemma, a local trivialization over $U$ is given by $(B, B_{\varphi_0}, A; \phi, \phi, \psi)$.
\end{proof}

In particular, \cref{ConditionLocallyTrivial} implies that the subalgebras $\varphi_x(A) \subset B$, $x \in U$, are all conjugate, in the sense that for each $x \in X$, there exists a smooth map $\phi': U \to \Aut(B)$ that maps $\varphi_x(A)$ to $\varphi_0(A)$.
We now give examples where this fails.

\begin{example}
\label{ex:loctrivfail1}
The following is an example of a 1-parameter family of pairwise non-conjugate subalgebras of $B = M_4(k)$, suggested in \cite{Cornulier}.
For $\lambda \in k$, consider the matrices
\begin{equation*}
  X =
  \begin{pmatrix} 
  0 & 1 & 0 & 0 \\
  0 & 0 & 0 & 0 \\
  0 & 0 & 0 & 1 \\
  0 & 0 & 0 & 0
  \end{pmatrix},
  \qquad
  Y_\lambda =
  \begin{pmatrix} 
  0 & 1 & 1 & 0 \\
  0 & 0 & 0 & \lambda \\
  0 & 0 & 0 & -\lambda \\
  0 & 0 & 0 & 0
  \end{pmatrix},
  \qquad
  Z =
  \begin{pmatrix} 
  0 & 0 & 0 & 1 \\
  0 & 0 & 0 & 0 \\
  0 & 0 & 0 & 0 \\
  0 & 0 & 0 & 0
  \end{pmatrix}
\end{equation*}
in $B$.
Let $A_\lambda$ be the subalgebra of $M_4(k)$ spanned by $X$, $Y_{\lambda}$, $Z$ and the identity matrix.
These subalgebras are pairwise non-isomorphic, except for $A_\lambda \cong A_{\lambda^{-1}}$, which can be seen by that fact that the system of equations 
\begin{equation*}
x^2 = 0, \qquad y^2 = 0, \qquad xy = \mu yx, \qquad xy\neq 0
\end{equation*}
in $A_{\lambda}$ has a solution if and only if $\mu \in \{\lambda, \lambda^{-1}\}$.
Now let $A$ be the algebra spanned by $1$ and elements $\tilde{X}$, $\tilde{Y}$, $\tilde{Z}$, $\tilde{Z}^\prime$, subject to the relation $\tilde{X}\tilde{Y} = \tilde{Z}^\prime$, $\tilde{Y}\tilde{X} = \tilde{Z}$ and all other products zero.
Let $\phi_\lambda : A \to A_\lambda \subset B$ be the algebra homomorphism sending $\tilde{X} \mapsto X$, $\tilde{Y} \to Y_\lambda$, $\tilde{Z} \to Z$, $\tilde{Z}^\prime \mapsto \lambda Z$.
Then $\phi : k \to \Hom(A, B)$ is a smooth family of homomorphisms with pairwise non-isomorphic image.
\end{example}

\begin{example}
\label{ex:loctrivfail2}
Consider the setting of \cref{ExampleTwistedModuleNotImplementing}.
There $B = M_{2m}(k)$ and $A = k \oplus V$, where $V$ is a $k$-vector space and the multiplication is such that the product of any two elements in $V$ is zero.
We then explained how any linear map $\tilde{\varphi} : V \to M_m(k)$ gives rise to an algebra homomorphism $\varphi: A \to B$.
Let now $\tilde{\varphi} : X \to \Hom(V, M_m(k))$ be a smooth family of \emph{injective} linear maps. 
Then the relation of \cref{LemmaLocallyTrivial} will be satisfied for some $x$ close to $x_0$ if and only if the subalgebras $\varphi_x(A)$ and $\varphi_{x_0}(A)$ are conjugate in $M_{2m}(k)$.
Since every automorphism of $M_{2m}(k)$ is inner, this implies that there exists $b \in \mathrm{GL}_{2m}(k)$ such that $b \varphi_x(V) b^{-1} = \varphi_{x_0}(V)$.
Each $\varphi_x(V)$ is a $\dim(V)$-dimensional subspace of the lower left corner $M_m(k) \subset M_{2m}(k)$, and (as conjugation by multiples of the identity in $\mathrm{GL}(k)$ does nothing), we are looking at the action of the projective linear group $\mathrm{PGL}_{2m}(k)$ on the Gra\ss mannian of $\dim(V)$-dimensional subspaces inside $M_m(k) \subset M_{2m}(k)$.
\\
As  observed in \cite{Cornuliera}, if we set $m=4$ and let $V = k^8$, then this Gra\ss mannian is a manifold of dimension $\dim(V)(m^2 - \dim(V)) = 64$, while $\mathrm{PGL}_{2m}(k)$ has dimension $(2m)^2 - 1 = 63$.
Hence the action cannot be locally transitive. 
Therefore choosing an $X$-family of subspaces transverse to the orbits of $\mathrm{PGL}_{2m}(k)$ and homomorphisms $\varphi$ with these subspaces as image produces an example where \eqref{ConditionLocallyTrivial} does not hold in any neighborhood of the point $x_0$.
\end{example}

In order to guarantee  the condition imposed by \cref{LemmaLocallyTrivial} one may restrict to  the full sub-bicategory
\begin{equation*}
\sssAlgBdlbi_{k}(X) \subset \sAlgBdlbi_k,
\end{equation*}
of \emph{semisimple algebra bundles}, the objects of which are algebra bundles with semisimple fibers.
Recall that by \cref{co:ssimp} \emph{all} bimodules between semisimple super algebras are implementing.
As announced at the beginning of this section, the bicategory of semisimple algebra bundles admits a framing.
This statement follows from the following result (which we will prove later in this section).

\begin{proposition} \label{PropositionLocallyTrivial}
Let $\mathcal{A}$ and $\mathcal{B}$ be semisimple super algebra bundles and let $\varphi: \mathcal{A} \to \mathcal{B}$ be a super algebra bundle homomorphism. 
Then, $\mathcal{B}_\varphi$ is a  $\mathcal{B}$-$\mathcal{A}$-bimodule bundle in the sense of \cref{def:bimodulebundle}.
\end{proposition}

\begin{corollary} \label{CorollaryFraming}
The bicategory of semisimple algebra bundles over $X$ has a framing
\begin{equation} \label{FramingForSemiSimple}
\sssAlgBdl_{k}(X) \to \sssAlgBdlbi_k(X)
\end{equation}
by the category of semisimple super algebra bundles and all super algebra bundle homomorphisms.
\end{corollary}

\begin{proof}
\Cref{PropositionLocallyTrivial} states that for a morphism $\varphi : \mathcal{A} \to \mathcal{B}$ in $\sssAlgBdl_k(X)$, we indeed have a well-defined  $\mathcal{B}$-$\mathcal{A}$-bimodule bundle $\mathcal{B}_\varphi$, which is then implementing by \cref{co:ssimp}. 
Hence, there is a well-defined functor $\sssAlgBdl_k(X) \to \sssAlgBdlbi_k(X)$.
That this is indeed a framing follows again from \cref{lem:fibrewiseinvertibility:b}.
\end{proof}

The remainder of this section is dedicated to proving \cref{PropositionLocallyTrivial}. Note that we only have to prove that $\mathcal{B}_{\varphi}$ is a bimodule bundle.
The crucial step in the proof is the following lemma.

\begin{lemma} \label{LemmaOrbit}
  Let $A$ and $B$ be semisimple super algebras.
  Then the orbits of the action of $B^\times_0$ on $\Hom(A, B)$ given by $(b, \varphi) \mapsto i(b) \circ \varphi$ are open and closed.
\end{lemma}

\begin{proof}
Let $\varphi \in \Hom(A, B)$. 
We have to show that each $\psi \in \Hom(A, B)$ lying in the same connected component as $\varphi$ is of the form $\psi = i(b) \circ \varphi$.
\\
Let $B = B_1 \oplus \dots \oplus B_n$ be the direct sum decomposition of $B$ into simple summands $B_i$ and let $p_1, \dots, p_n$ be the projections onto $B_i$, which are central idempotents satisfying $p_i p_j = 0$ for $i \neq j$ and $p_1 + \dots + p_n = 1$. 
Define $\varphi_i : A \to B_i$ by $\varphi_i(a) = p_i \varphi(a) p_i$.
Then each $\varphi_i$ is an algebra homomorphism and $\varphi = \varphi_1 + \dots + \varphi_n$.
Now if given $\psi \in \Hom(A, B)$, there exist elements $b_i \in (B_i)_0^\times$, $i=1, \dots, n$, such that $\psi_i = i(b_i) \circ \varphi_i$, then the element $b := b_1 + \dots + b_n \in B^\times_0$ satisfies $\psi = i(b) \circ \varphi$.
Hence we can and will from now on assume that $B$ is simple.
\, \\
Let $A = A_1 \oplus \dots \oplus A_m$ be the decomposition of $A$ into simple summands and let $p_1, \dots, p_m$ be the corresponding projections.
As each $A_i$ is simple, the ideal $\ker(\varphi) \subset A$ must be of the form $p A p$ for some central idempotent $p$ in $A$ of the form $p = p_{i_1} + \dots + p_{i_r}$, with $1 \leq i_1 < \dots < i_r \leq m$.
As the set of these idempotents is discrete, each $\psi \in \Hom(A, B)$ in the same connected component as $\varphi$ must satisfy $\ker(\psi) = \ker(\varphi)$.
Therefore, by passing to the quotient $A / pAp \cong \bigoplus_{j \neq i_1, \dots, i_r} A_j$, which is again semisimple, we can and will assume from now on that $\varphi$ is injective (and hence also each homomorphism $\psi$ in the same connected component as $\varphi$).
\\
Let $q_i = \varphi(p_i)$, an idempotent in $B$. 
For $\psi \in \Hom(A, B)$ in the same connected component as $\varphi$, let $q_i^\prime = \psi(p_i)$. 
Then if $\psi$ is sufficiently close to $\varphi$, \cref{LemmaCloseProjections} provides invertible elements $a_1, \dots, a_m \in B^\times_0$ such that $q_i^\prime = a_i q_i a_i^\prime$.
Then $a = \sum_{j=1}^m q^\prime_j a_j q_j$ is invertible with  inverse $a^{-1} = \sum_{j=1}^m q_j a_j^{-1} q^\prime_j$  and for each $i=1, \dots, m$
\begin{equation*}
  a q_i a^{-1} = \sum_{jk=1}^m q^\prime_j a_j q_j q_i q_k a_k^{-1} q^\prime_k = q^\prime_i a_i q_i a_i^{-1} q^\prime_i = q^\prime_i.
\end{equation*}
In general, if $\psi$ is in the same connected component as $\varphi$, we obtain an invertible $a \in B^\times_0$ by applying \cref{LemmaCloseProjections} to subsequent subdivisions of a connecting path, as in the proof of \cref{PropositionInvertibleImplementing}.
Hence replacing $\psi$ by $i(a^{-1}) \circ \psi$ reduces to the case $q_i = q_i^\prime$.
\\
Now, for each $i=1, \dots, m$, the algebra $B_i := q_i B q_i$ is again simple. 
Indeed, if $I \subset B_i$ is an ideal, then 
\begin{equation*}
I^\prime := I + IA(1-q_i) + (1-q_i) A I + (1-q_i) AIA (1-q_i)
\end{equation*}
 is an ideal in $B$, which is non-trivial if and only if $I$ is.
Hence as $B$ does not have any non-trivial ideals, neither does $B_i$.
Therefore $\varphi : A \to B$ splits up into a direct sum of homomorphisms $\varphi_i : A_i \to B_i$, where both $A_i$ and $B_i$ are simple, and the same holds for $\psi : A \to B$ close enough to $\varphi$.
Assume first that $B_i$ is also central simple.
Then by the graded version of the Skolem-Noether theorem \cite[Thm.~3.2]{JaberFirstKind}, there exist $b_i \in B_i$ such that $\psi_i = i(b_i) \circ \varphi_i$.
We then conclude that $b := b_1 + \dots + b_m$ is invertible and satisfies $\psi = i(b) \circ \varphi$.
\\
We are left to discuss the case that one of the simple summands $B_i$ is not central simple.
This can only happen if $k=\R$ and $Z(B_i) = \C$, as $Z(B_i)$ is always a finite field extension of $k$ and $\C$ does not have any non-trivial finite field extensions.
In that case, we replace $A_i$ by $\tilde{A}_i = A_i \otimes_{\R} \C$ and extend homomorphisms $\phi : A_i \to B_i$ to homomorphisms $\tilde{\phi} : \tilde{A}_i \to B_i$ by complex linearity.
Observe that if we find $b_i \in B_i$ such that $\tilde{\psi}_i = i(b_i) \circ \tilde{\varphi}_i$, then also ${\psi}_i = i(b_i) \circ {\varphi}_i$.
Now $B_i$ is central as an algebra over $\C$, however $\tilde{A}_i$ is still semisimple but may no longer be simple (for example, if $A_i = \C$, then $\tilde{A}_i = \C \otimes_\R \C \cong \C \oplus \C$).
In that case, we have to go back and split $\tilde{A}_i = \tilde{A}_{i, 1} \oplus \dots \oplus \tilde{A}_{i, k}$ into simple summands and correspondingly $B_i = B_{i_1} \oplus \dots \oplus B_{i, k}$. 
However, as we are now over $k = \C$, each algebra $B_{i, j}$ is now necessarily central simple.
\end{proof}

\begin{proof}[Proof of \cref{PropositionLocallyTrivial}]
As this is a local question, we may assume that $\mathcal{A} = \underline{A}$ and $\mathcal{B} = \underline{B}$ are trivial bundles and that $\mathcal{B}_\varphi = \underline{B}_\varphi$ for some smooth function $\varphi: X \to \Hom(A, B)$.
We have to find trivializations in a neighborhood of each point in $X$. 
\\
To this end, fix $x_0 \in X$ and consider the map $B_0^\times \to \Hom(A, B)$, $b \mapsto i(b) \circ \varphi_{x_0}$.
The stabilizer $H$ of $\varphi_{x_0}$ is the set of all $b \in B_0^\times$ that commute with the subalgebra $\varphi_{x_0}(A)$.
Hence by \cref{LemmaOrbit}, the map $[b] \to i(b) \circ \varphi_x$ provides an identification of the homogeneous space $B_0^\times /H$ with the connected component $\Hom(A, B)_{\varphi_{x_0}}$ of $\varphi_{x_0}$ in $\Hom(A, B)$.
In particular, the map $B_0^\times \to \Hom(A, B)_{\varphi_{x_0}}$ is a fiber bundle (with typical fiber $H$).
Hence in a neighborhood $U$ of $x_0$, we can choose a lift $b : X \to B_0^\times$ of $\varphi : X\to \Hom(A, B)$.
By definition, such a lift satisfies $b_x \varphi_{x_0}(a) b_x^{-1} = \varphi_x(a)$ for all $x \in U$ and $a \in A$.
Then $(B, B_{\varphi_{x_0}}, A; \mathrm{id}, r_b, \varphi)$ provides a local trivialization over $U$, where $r_b$ is the right multiplication by $b$.
\end{proof}

Another consequence of \cref{LemmaOrbit} is the following result, that concerns the possible alternative definition of bimodule bundles mentioned in \cref{RemarkNoTypicalFiber}. It shows that it is -- for semisimple algebras -- equivalent to our definition.

\begin{proposition}
\label{prop:weakernotioncoincides}
Let $\mathcal{A}$ and $\mathcal{B}$ be semisimple super algebra bundles, and let $\mathcal{M}$ be a super vector bundle together with a super algebra bundle homomorphism $\varphi : \mathcal{A} \otimes \mathcal{B}^{\opp} \to \mathrm{End}(\mathcal{M})$. Then, $\mathcal{M}$ canonically has the structure of an $\mathcal{A}$-$\mathcal{B}$-bimodule bundle in the sense of \cref{def:bimodulebundle}.   
\end{proposition}

\begin{proof}
The fibers $\mathcal{M}_x$, $x \in X$, obtain bimodule structures by setting
\begin{equation*}
  a \lact m \ract b = (-1)^{|b||m|} \varphi(a \otimes b)m.
\end{equation*}
We have to construct local trivializations for this family of bimodules.
To this end, let $\phi : \mathcal{A}|_U \to U \times A$, $\psi: \mathcal{B}|_U = U \times B$ be local trivializations of superalgebra bundles and let $g : \mathcal{M}|_U = U \times M$ be a vector bundle trivialization.
Under these trivializations, the super algebra bundle homomorphism $\mathcal{A} \otimes \mathcal{B}^{\opp} \to \mathrm{End}(\mathcal{M})$ corresponds to a smooth family $\tilde{\varphi} : X \to \Hom(A \otimes B^{\mathrm{op}}, \End(M))$.
Let $x_0 \in X$ and consider the map $\GL_0(M) \to \Hom(A \otimes B^{\mathrm{op}}, \End(M))$ given by $u \mapsto i(u) \circ \tilde{\varphi}_{x_0}$.
    As both ${A} \otimes {B}^{\opp}$ and $\End(M)$ are semisimple, Lemma~\ref{LemmaOrbit} implies that the image of this map is open and closed.
    As in the proof of \cref{PropositionLocallyTrivial}, there exists a neighborhood $U^\prime \subset U$ of $x_0$ on which we can lift $\varphi$ to a map $u: U^\prime \to \GL_0(M)$ such that $\tilde{\varphi}_x = i(u_x) \circ \tilde{\varphi}_{x_0}$.
    Then the map $u^{-1} g : \mathcal{M}|_{U^\prime} \to U^\prime \times M$ is an intertwiner along $\phi$ and $\psi$, for the bimodule action on $M$ given by $\tilde{\varphi}_{x_0}$.
This provides the desired local trivialization of $\mathcal{M}$ as a bimodule.
\end{proof}

\subsection{Symmetric monoidal structures} \label{SectionSymmetricMonoidalStructures}

Our bicategory $\sAlgBdlbi_k$ of  super algebra bundles over $X$ (\cref{DefinitionPrestwoVectBdl}) is a bundle version of the bicategory $\sAlgbiimp_k$ of implemented algebras. Since
 $\sAlgbiimp_k$ is not monoidal, as seen in \cref{sec:smsonimplementedalgebras},  the bicategory $\sAlgBdlbi_k(X)$ cannot be monoidal either.

To explore this further, let us first note that the notion of tensor product of two super algebra bundles $\mathcal{A}$ and $\mathcal{B}$ over a manifold $X$ is unproblematic: 
The fiberwise tensor products $\mathcal{A}_x \otimes \mathcal{B}_x$, $x \in X$, form again a super algebra bundle $\mathcal{A} \otimes \mathcal{B}$; and if   $\mathcal{A}$ and $\mathcal{B}$ have typical fibers $A$ and $B$, then $\mathcal{A} \otimes \mathcal{B}$ has typical fiber $A \otimes B$. 
It is also unproblematic to define the exterior product of bimodule bundles: if an $\mathcal{A}$-$\mathcal{B}$-bimodule bundle $\mathcal{M}$ and an $\mathcal{A}'$-$\mathcal{B}'$-bimodule bundle $\mathcal{M}'$ have typical fibres $(A,M,B)$ and $(A',B',M')$, respectively, then $\mathcal{M} \otimes \mathcal{M}'$ has typical fibre $(A\otimes A',M \otimes M',B \otimes B')$. 
However, as seen in \cref{sec:smsonimplementedalgebras} (the bimodule from \cref{ExampleImplementing3} is a non-implementing  exterior product of two implementing bimodules), the bimodule bundle $\mathcal{M} \otimes \mathcal{M}'$ is not necessarily implementing. 
Since this is a purely fibrewise problem, the solution in \cref{sec:smsonimplementedalgebras} also works in the bundle version, and we have the following result.

\begin{proposition}
\label{prop:symmetricmonoidalstructures}
The sub-bicategory  $\sssAlgBdlbi_k(X)$ and the sub-bigroupoid $\sAlgBdlbigrpd kX$ of $\sAlgBdlbi_k(X)$ are symmetric monoidal.
\end{proposition}

\begin{remark}
One could also consider the full subcategory of $\sAlgBdlbi_k(X)$ consisting of those super algebra bundles whose typical fibers $A$ satisfy only $\mathrm{HH}^1(A) = 0$, as by \cref{PropHochschildImplementing}, all bimodules between such algebras are implementing.
Moreover, by the tensor product formula for Hochschild cohomology (see, e.g., \cite{HochschildTensorProduct}), we have
\begin{equation*}
  \mathrm{HH}^1(A \otimes B) = (\mathrm{HH}^0(A) \otimes \mathrm{HH}^1(B)) \oplus (\mathrm{HH}^1(A) \otimes \mathrm{HH}^0(B)) = 0,
\end{equation*}
provided $\mathrm{HH}^1(A) = \mathrm{HH}^1(B) = 0$.
Hence, the tensor product of objects is well defined inside this bicategory.
However, it seems to us  that the category of semisimple super algebras is more natural to consider, as (a) our notion of bimodule bundle coincides with the weaker notion of \cref{RemarkNoTypicalFiber} in this case (see \cref{prop:weakernotioncoincides}); and (b) this category admits a larger framing, as discussed in \cref{PropositionLocallyTrivial} and \cref{CorollaryFraming}.
\end{remark}

We infer the following result about the invertible objects.

\begin{proposition}
\label{lemma:invertibleobjectsinbundles}
In both symmetric monoidal bicategories, $\sAlgBdlbigrpd kX$ and $\sssAlgBdlbi_k(X)$, a super algebra bundle is invertible if and only if it is fibrewise invertible, i.e., if and only if all its fibres are central simple super algebras.
\end{proposition}

\begin{proof}
It is clear that if a super algebra bundle $\mathcal{A}$ is invertible, then its fibres must be invertible super algebras. 
Conversely, if $\mathcal{A}$ is a super algebra bundle with central simple fibres, then taking fibrewise opposite algebras results in a super algebra bundle $\mathcal{A}^{\opp}$. 
For each $x \in X$,  $\mathcal{A}_x \otimes \mathcal{A}_x^{\opp}$ is Morita equivalent to $k$, \ie $\mathcal{A}^{\opp}_x$ is an inverse to $\mathcal{A}_x$.
The bimodules implementing the Morita equivalence
form bimodule bundles (both have $\mathcal{A}$ as their underlying vector bundle). As they are fibrewise invertible, they are invertible as bundles, by \cref{lem:fibrewiseinvertibility:c}.
This shows that $\mathcal{A}^{\opp}$ is an inverse of $\mathcal{A}$ with respect to the tensor product. 
\end{proof}

\begin{corollary}
The following table describes the dualizable, fully dualizable, and invertible objects in both symmetric monoidal categories of super algebra bundles: 
\begin{center}
\begin{tabular}{l|c|c|c}
 & dualizable & fully dualizable & invertible \\\hline
$\sAlgBdlbigrpd kX$ & central simple & central simple & central simple \\
$\sssAlgBdlbi_k(X)$ & all & all & central simple \\
\end{tabular}
\end{center}
\end{corollary}

\begin{proof}
In the first line, being dualizable or fully dualizable implies already fibrewise invertibility. In the second line, the argument in the proof of \cref{lemma:invertibleobjectsinbundles} in combination with \cref{lem:fibrewiseinvertibility:b} shows that every semisimple super algebra bundle is fully dualizable. 
\end{proof}

In order to discuss the invertible super algebra bundles, we introduce the full sub-bicategory
\begin{equation}
\cssAlgBdlbi_k(X) \subset \sssAlgBdlbi_k(X) \subset \sAlgBdlbi_k(X)
\end{equation}
over all central simple super algebra bundles, which is then symmetric monoidal. 
A result of Donovan-Karoubi \cite[Theorems 6 and 11]{DK70} shows the following.

\begin{theorem}
\label{th:classalg}
There is a canonical bijection
\begin{equation*}
\mathrm{h}_0(\cssAlgBdlbi_k(X)) \cong {\mathrm{H}}^0(X,\mathrm{BW}_k) \times  {\mathrm{H}}^1(X,\Z_2) \times \mathrm{Tor}(\check{\mathrm{H}}^2(X,\underline{k}^{*}))\text{,}
\end{equation*}
where $\check{\mathrm{H}}^2$ denotes  \v{C}ech cohomology, and $\underline{k}^{*}$ is the sheaf of smooth $k^{*}$-valued functions. Moreover, this bijection becomes an isomorphism of groups upon defining the group structure on the right hand side by 
\begin{equation} \label{GroupStructureDonovanKaroubi}
(\alpha_0,\alpha_1,\alpha_2) \cdot (\beta_0,\beta_1,\beta_2) := (\alpha_0+\beta_0,\alpha_1+\beta_1,(-1)^{\alpha_1\cup\beta_1}\alpha_2\beta_2)\text{.}
\end{equation}
\end{theorem}

\begin{remark}
\begin{enumerate}[(1)]

\item
The cup product is viewed here as a map
\begin{equation*}
{\mathrm{H}}^1(X,\Z_2) \times {\mathrm{H}}^1(X,\Z_2) \to {\mathrm{H}}^2(X,\Z_2) \to \mathrm{Tor}(\check{\mathrm{H}}^2(X,\underline{k}^{*})) 
\end{equation*} 
with the second arrow induced by the inclusion $\Z_2 \to k^{*}: x \mapsto (-1)^{x}$. 

\item 
\label{re:torsion}
We have $\mathrm{Tor}(\check{\mathrm{H}}^2(X,\underline{\C}^{*}))=\mathrm{Tor}(\mathrm{H}^3(X,\Z))$ and $\mathrm{Tor}(\check{\mathrm{H}}^2(X,\underline{\R}^{*}))=\mathrm{H}^2(X,\Z_2)$.

\end{enumerate}
\end{remark}

\begin{proof}[Proof (sketch)]
Let us briefly recall how the classification in \cref{th:classalg} comes about. 
Given a super algebra bundle $\mathcal{A}$ over $X$, one chooses an open cover $\{U_{\alpha}\}$ of $X$ admitting local trivializations $\varphi_{\alpha}:\mathcal{A}|_{U_{\alpha}} \to X \times A_{\alpha}$. 
As mentioned above, the typical fibres $A_{\alpha}$ can be assumed to be constant on the connected components of $X$; this defines the class in ${\mathrm{H}}^0(X,\mathrm{BW}_k)$. 
On a double intersection, the local trivializations determine a smooth map $\varphi_{\alpha\beta}:U_{\alpha}\cap U_{\beta} \to \mathrm{Aut}(A)$, where $A=A_{\alpha}=A_{\beta}$. 
At $x\in U_{\alpha}\cap U_{\beta}$, we obtain the irreducible $A$-$A$-bimodule $A_{\varphi_{\alpha\beta}(x)}$. Since $A$ is central and simple, we have either $A_{\varphi_{\alpha\beta}(x)}\cong A$ or $A_{\varphi_{\alpha\beta}(x)}\cong \Pi A$ (\cref{re:classbimodcsa}). 
Thus,  there exist $\varepsilon_{\alpha\beta}\in \Z_2$ and smooth maps $f_{\alpha\beta}:U_{\alpha}\cap U_{\beta} \to \underline{\mathrm{GL}}(A)$ such that $f_{\alpha\beta}(x): \varepsilon_{\alpha\beta}A \to A_{\varphi_{\alpha\beta}(x)}$ is  an invertible even intertwiner. 
The cocycle condition $\varphi_{\alpha\gamma}=\varphi_{\beta\gamma}\circ \varphi_{\alpha\beta}$ implies $\varepsilon_{\alpha\gamma}=\varepsilon_{\alpha\beta} + \varepsilon_{\beta\gamma}$; thus, $\varepsilon_{\alpha\beta}$ is a \v Cech 1-cocycle defining the class in ${\mathrm{H}}^1(X,\Z_2)$. 
Finally, the same cocycle condition implies that $f_{\alpha\gamma}$ and $f_{\beta\gamma}\circ f_{\alpha\beta}$  differ by a unique smooth scalar function $x_{\alpha\beta\gamma}:U_{\alpha}\cap U_{\beta}\cap U_{\gamma} \to k^{\times}$, which in turn is a \v Cech 2-cocycle defining the class in $\check{\mathrm{H}}^2(X,\underline{k}^{*})$. 
\end{proof}

\begin{remark}
Both symmetric monoidal bicategories $\sAlgBdlbigrpd kX$ and $\sssAlgBdlbi_k(X)$ have a second symmetric monoidal structure given by the direct sum of algebra bundles, and the exterior direct sum of bimodule bundles. The two monoidal structures are compatible with each other in the sense of distributive laws, and form a \emph{commutative rig bicategory}.
\end{remark}

\begin{remark}
In the ungraded case, we consider the sub-bicategories $\AlgBdlbigrpd kX$ and $\ssAlgBdlbi_k(X)$.  They are symmetric monoidal and framed by $\AlgBdlgrpd kX$ and $\ssAlgBdl_k(X)$, respectively. 
The invertible objects form a full sub-bicategory $\csAlgBdlbi_k(X)$.   The ungraded version of \cref{th:classalg} was also proved by Donovan-Karoubi \cite[Theorems 3 and 8]{DK70}, and gives  a group isomorphism
\begin{equation*}
\mathrm{h}_0(\csAlgBdlbi_k(X)) \cong  \mathrm{H}^0(X,\mathrm{Br}_k)\times \mathrm{Tor}(\check{\mathrm{H}}^2(X,\underline{k}^{*}))\text{,}
\end{equation*}  
with the direct product group structure on the right hand side.
\end{remark}

\subsection{Pre-2-stacks of  algebra bundles}
\label{sec:2stack2vec}

So far we have worked over a fixed smooth manifold $X$. In order to consider all smooth manifolds at the same time, we use the notion of a presheaf of (bi-)categories on the category $\Mfd$ of smooth manifolds.
For example, super algebra bundles and ordinary super algebra bundle homomorphisms form a presheaf
\begin{equation*}
X \mapsto \sAlgBdl_k(X)
\end{equation*}
of symmetric monoidal categories, which we will  denote by $\sAlgBdlbi_k$. 

Questions of descent are discussed w.r.t.~a Grothendieck topology on $\Mfd$, turning it into a \emph{site}. In our case, it will be the Grothendieck topology of surjective submersions. To any presheaf of categories $\mathscr{F}$ and any surjective submersion $\pi:Y \to X$ one associates a category $\Des_{\mathscr{F}}(\pi)$ of \emph{descent data}, in which the objects are pairs $(\mathcal{A},\phi)$ of an object $\mathcal{A}$ of $\mathscr{F}(Y)$ and an isomorphism $\phi:\pr_2^{*}\mathcal{A} \to \pr_1^{*}\mathcal{A}$ over $Y^{[2]}=Y \times_X Y$ satisfying the cocycle condition $\pr_{23}^{*}\phi \circ \pr_{12}^{*}\phi = \pr_{13}^{*}\phi$. Pullback along $\pi$ establishes a functor
\begin{equation}
\label{eq:functorR}
R_{\pi}:\mathscr{F}(X) \to \Des_{\mathscr{F}}(\pi)\text{.}
\end{equation}
We recall that $\mathscr{F}$ is called \emph{prestack} if $R_{\pi}$ is full and faithful for all surjective submersions $\pi$, and it is called \emph{stack} if $R_{\pi}$ is an equivalence of categories. It is well-known and straightforward to show that the presheaf $\sAlgBdl_k$ is a stack on the site $\Mfd$.

Similarly,
\begin{equation*}
X\mapsto \sAlgBdlbi_k(X)
\end{equation*}
forms a \emph{presheaf of bicategories} over $\Mfd$, which we denote by $\sAlgBdlbi_k$. It has three important sub-presheaves:
\begin{enumerate}

\item 
The sub-presheaf $\sAlgBdlbigrpd k-$ where only invertible bimodule bundles and invertible intertwiners are admitted. This is a presheaf of symmetric monoidal framed bigroupoids. 

\item
The sub-presheaf $\sssAlgBdlbi_k$ where only the semisimple super algebra bundles are admitted. This is a presheaf of symmetric monoidal framed bicategories. 

\item
The further sub-presheaf $\cssAlgBdlbi_k$, where only the central simple super algebra bundles are admitted.   

\end{enumerate}
Moreover, we have ungraded versions of all these four presheaves, and thus a collection of not less than eight presheaves of bicategories.

For presheaves of \emph{bi}categories $\mathscr{F}$, there is a analogous \emph{bi}category $\Des_{\mathscr{F}}(\pi)$ of descent data, and a corresponding functor $R_{\pi}$ as in \cref{eq:functorR}. Then, $\mathscr{F}$ is called \emph{pre-2-stack} if the functor $R_{\pi}$ is full and faithful, i.e., if it induces, for all objects $\mathcal{A},\mathcal{B}$ of $\mathscr{F}(X)$, an equivalence
\begin{equation*}
\Homcat_{\mathscr{F}(X)}(\mathcal{A},\mathcal{B}) \cong \Homcat_{\Des_{\mathscr{F}}(\pi)}(R_{\pi}(\mathcal{A}),R_{\pi}(\mathcal{B}))\text{.} 
\end{equation*} 
Moreover, $\mathscr{F}$ is called \emph{2-stack} if $R_{\pi}$ is an equivalence for all surjective submersions $\pi$. For a general proper discussion of  (pre-)2-stacks we refer to \cite{nikolaus2}.

\begin{proposition}
\label{lem:pre2stack}
All eight presheaves of bicategories of algebras collected above  
are pre-2-stacks.
\end{proposition}

\begin{proof}
We perform the proof in case of the presheaf $\mathscr{F}:=\AlgBdlbi_{k}$ of algebra bundles, the other cases are analogous. 
We also assume without loss of generality that the surjective submersion $\pi:Y \to M$ if of the form where $Y$
is the disjoint union of open sets $U_{\alpha}$ that cover $X$.  
Now consider two algebra bundles $\mathcal{A}$ and $\mathcal{B}$ over $X$. 
An object in the category $\Homcat_{\Des_{\mathscr{F}}(\pi)}(R_{\pi}(\mathcal{A}),R_{\pi}(\mathcal{B}))$ is  a collection $\mathcal{M}_{\alpha}$ of implementing $\mathcal{A}|_{U_{\alpha}}$-$\mathcal{B}|_{U_{\beta}}$-bimodule bundles, together with a collection $\phi_{\alpha\beta}: \mathcal{M}_{\alpha} \to \mathcal{M}_{\beta}$ of invertible intertwiners over the intersections $U_{\alpha}\cap U_{\beta}$, satisfying the cocycle condition over triple intersections. A morphism $(\mathcal{M}_{\alpha},\phi_{\alpha\beta}) \to (\mathcal{M}'_{\alpha},\phi'_{\alpha\beta})$ is a collection $\psi_{\alpha}: \mathcal{M}_{\alpha} \to \mathcal{M}'_{\alpha}$ of intertwiners such that 
\begin{equation*}
\phi'_{\alpha\beta} \circ \psi_{\alpha} = \psi_{\beta} \circ \phi_{\alpha\beta}\text{.}
\end{equation*} 
The functor $R_{\pi}$ induces the functor
\begin{equation*}
R_\pi(\mathcal{A},\mathcal{B}): \sBimodBdlimp_{\mathcal{A},\mathcal{B}}(X) \rightarrow \Homcat_{\Des_{\mathscr{F}}(\pi)}(R_{\pi}(\mathcal{A}),R_{\pi}(\mathcal{B}))\text{,}
\end{equation*}
which sends a globally defined implementing $\mathcal{A}$-$\mathcal{B}$-bimodule bundle $\mathcal{M}$ to the collection $\mathcal{M}_\alpha := \mathcal{M}|_{U_\alpha}$, together with the canonical identifications $\phi_{\alpha\beta} : \mathcal{M}_\alpha|_{U_\alpha \cap U_\beta} = \mathcal{M}|_{U_\alpha \cap U_\beta} = \mathcal{M}_\beta |_{U_\alpha \cap U_\beta}$. It sends an intertwiner $\psi : \mathcal{M} \to \mathcal{M}'$ between globally defined implementing bimodule bundles to the collection $\psi_{\alpha} := \psi|_{U_{\alpha}}$. 
We have to show that $R_\pi(\mathcal{A},\mathcal{B})$ is an equivalence of categories. 
\\
It is clear that $R_\pi(\mathcal{A},\mathcal{B})$ is fully faithful, basically because bundle morphisms form a sheaf. In order to show that 
$R_\pi(\mathcal{A},\mathcal{B})$ is essentially surjective, we have to show that for any  descent object $(\mathcal{M}_{\alpha},\phi_{\alpha\beta})$  
 there exists a globally defined implementing $\mathcal{A}$-$\mathcal{B}$-bimodule bundle $\mathcal{M}$, together with locally defined invertible intertwiners $\gamma_{\alpha}: \mathcal{M}|_{U_{\alpha}} \to \mathcal{M}_{\alpha}$, such that $\phi_{\alpha\beta}\circ \gamma_{\alpha}=\gamma_{\beta}$; 
in other words, the locally defined bimodule bundles $\mathcal{M}_{\alpha}$ can be glued together. 
Since vector bundles form a stack, it is clear that $\mathcal{M}$ and $\gamma_{\alpha}$ exist as a vector bundle and vector bundle isomorphisms, respectively. Then, the bimodule actions and their local triviality in the sense of \cref{def:bimodulebundle} can be transferred along $\gamma_{\alpha}$ to $\mathcal{M}$. Implementability is a fibrewise property and hence holds for $\mathcal{M}$ just as for $\mathcal{M}_{\alpha}$.
\end{proof}

We close this article, by arguing that six of our eight pre-2-stacks are \emph{not} 2-stacks.
The remaining two pre-2-stacks are in fact 2-stacks, but the proof is deferred to \cite{Kristel2020}.
Let us first describe the assumption on a pre-2-stack that turns it into a 2-stack; namely, that the functor $R_{\pi}$ is essentially surjective, in case of the pre-2-stack $\AlgBdlbi_\C$. We consider again an open cover $\{U_{\alpha}\}_{\alpha\in A}$ of a smooth manifold $X$, and an object in $\Des_{\AlgBdlbi_{\C}}(\pi)$, where $\pi:Y \to X$ is the surjective submersion obtained from the disjoint union of the open sets $U_{\alpha}$. Such an object consists of a family $\mathcal{A}_{\alpha}$ of algebra bundles over $U_{\alpha}$, together with invertible (hence automatically implementing)  $\mathcal{A}_{\beta}$-$\mathcal{A}_{\alpha}$-bimodule bundles $\mathcal{M}_{\alpha\beta}$ over $U_{\alpha}\cap U_{\beta}$, and invertible intertwiners $\mu_{\alpha\beta\gamma}: \mathcal{M}_{\beta\gamma} \otimes_{\mathcal{A}_{\beta}} \mathcal{M}_{\alpha\beta} \to \mathcal{M}_{\alpha\gamma}$ over $U_{\alpha}\cap U_{\beta}\cap U_{\gamma}$ that satisfy the obvious associativity condition over 4-fold intersections. If the pre-2-stack $\AlgBdlbi_{\C}$ was a 2-stack, then there must exist a globally defined algebra bundle $\mathcal{A}$ together with invertible $\mathcal{A}|_{U_\alpha}$-$\mathcal{A}_\alpha$-bimodule bundles $\mathcal{N}_\alpha$ and invertible intertwiners 
\begin{equation*}
  \varphi_{\alpha \beta} : \mathcal{N}_\beta \otimes_{\mathcal{A}_{\beta}}  \mathcal{M}_{\alpha \beta} \longrightarrow \mathcal{N}_\alpha
\end{equation*}
that are compatible with the isomorphisms $\mu_{\alpha \beta \gamma}$ in an obvious way. However, this is not necessarily the case, as the following example shows. 
 
\begin{example}
\label{ex:nota2stack}
Suppose $X$ is a smooth manifold such that there exists a non-torsion element $\xi$ in $\h^3(X,\Z)$; for instance, $X=S^3$. Let $\mathcal{G}$ be a bundle gerbe whose Dixmier-Douady class is $\xi$, see \cite{Murray1996}. Such a bundle gerbe can be represented in terms of an open cover  $\{U_{\alpha}\}_{\alpha\in A}$ and complex line bundles $\mathcal{L}_{\alpha\beta}$ over $U_{\alpha}\cap U_{\beta}$, together with line bundle isomorphisms $\mu_{\alpha\beta\gamma}: \mathcal{L}_{\beta\gamma} \otimes \mathcal{L}_{\alpha\beta} \to \mathcal{L}_{\alpha\gamma}$ over triple intersections, satisfying a cocycle condition over four-fold intersections. Since line bundles are invertible $\underline{\C}$-$\underline{\C}$-bimodule bundles, putting $\mathcal{A}_{\alpha}:= \underline{\C}$ produces an object $(\mathcal{A}_{\alpha},\mathcal{L}_{\alpha\beta},\mu_{\alpha\beta\gamma})$ in $\Des_{\AlgBdlbi_{\C}}(\pi)$. However, this object does not correspond to any globally defined algebra bundle over $X$. Indeed, this would be an algebra bundle $\mathcal{A}$ over $X$ together with invertible $\mathcal{A}$-$\underline{\C}$-bimodule bundles $\mathcal{N}_{\alpha}$  and invertible intertwiners $\varphi_{\alpha\beta}:\mathcal{N}_{\beta} \otimes \mathcal{L}_{\alpha\beta} \to \mathcal{N}_{\alpha}$ compatible with $\mu_{\alpha\beta\gamma}$. Forgetting the left $\mathcal{A}$-action, this structure is known as a  bundle gerbe module for $\mathcal{G}$, and it is finite-dimensional and non-zero. As proved in \cite[Prop. 4.1]{Bouwknegt2002}, this means that the Dixmier-Douady class $\xi$ is torsion, a contradiction. 
\end{example}

It is now straightforward to see that also $\AlgBdlbigrpd \C-$ and $\csAlgBdlbi_{\C}$, as well as  $\sAlgBdlbi_{\C}$, $\sAlgBdlbigrpd \C-$, $\cssAlgBdlbi_{\C}$ are not 2-stacks. For algebra bundles over $\R$, the situation is different, because in that case bundle gerbes always have torsion classes. Indeed, we show in \cite{Kristel2020} that $\cssAlgBdlbi_{\R}$ and $\csAlgBdlbi_{\R}$ are indeed 2-stacks. 
The interested reader is invited to follow our work in \cite{Kristel2020}, where we continue this discussion from the perspective of 2-vector bundles.

\begin{appendix}

\setsecnumdepth{1}

\section{Picard-surjective super algebras}

\setsecnumdepth{1}

\label{sec:pssa}

In this appendix we provide two results about Picard-surjective super algebras that we have not been able to find in the literature. The first is about the preservation of Picard-surjectivity under the tensor product.

\begin{lemma}
\label{lem:tensorproductpicard}
If $A$ and $B$ are central simple Picard-surjective super algebras, then $A \otimes B$ is again Picard-surjective.
\end{lemma}

\begin{proof}
$A \otimes B$ is again central simple, hence every invertible $(A \otimes B)$-bimodule is isomorphic to either $A \otimes B$ or $\Pi(A \otimes B)$. 
Therefore we only need to show that $\Pi(A \otimes B)$ is induced by an automorphism.
Since $B$ is Picard-surjective, there exists $\varphi \in \Aut(B)$ is such that $\Pi B \cong B_\varphi$ as $B$-$B$-bimodules. 
One now checks that the map $\id \otimes \rho$ provides an isomorphism $\Pi(A\otimes B) \cong (A \otimes B)_{1 \otimes \varphi}$,
where $\rho: \Pi B \to B_\varphi$ is any isomorphism of $B$-$B$-bimodules.
\end{proof}

Being Picard-surjective is not a Morita invariant. However, our second result is the following, whose proof is based on \cite{Rickard}. 

\begin{proposition}
\label{prop:picsur}
Every Morita equivalence class of super algebras contains a Picard-surjective super algebra. 
\end{proposition}

The proof uses the following lemma.

\begin{lemma} \label{Lemma:PicardSurjective}
Let $A$ be a super algebra and suppose that, as a left module over itself, it is a direct sum $A = P_1 \oplus \cdots \oplus P_n$, where $P_1, \dots, P_n$ are indecomposable finitely generated projective left $A$-modules, one from each isomorphism class. Then $A$ is Picard-surjective.
\end{lemma}

\begin{proof}
Let $R$ be an invertible $A$-$A$-bimodule. We want to show that $M \cong A_\varphi$ for some $\varphi \in \mathrm{Aut}(A)$. 
To this end, the functor
\begin{equation*}
F : \Bimod_{A-k} \to \Bimod_{A-k},\qquad M \mapsto R \otimes_A M
\end{equation*}
which is an additive autoequivalence, as $R$ is invertible. Therefore, it preserves the properties of being indecomposable, finitely generated or projective.
It follows that the functor only permutes the modules $P_1, \dots, P_n$, up to isomorphism. Explicitly, there exists some permutation $\sigma \in S_n$ such such that $P_i \cong P_{\sigma_i}$ for each $i=1, \dots, n$. 
We obtain that as left $A$-modules, 
\begin{equation*}
R \cong R \otimes_A A =  F(A) \cong P_{\sigma_1} \oplus \dots \oplus P_{\sigma_n} \cong P_1 \oplus \dots \oplus P_n \cong A 
\end{equation*}
The right $A$-module structure on $R \cong A$ is given by some algebra homomorphism
\begin{equation*}
\varphi:  A^{\opp} \to  \underline{\End}_{A-k}(R) \cong \underline{\End}_{A-k}(A) \cong A^{\opp}.
\end{equation*}
An algebra homomorphism $A^{\opp} \to A^{\opp}$ is the same as an algebra homomorphism $A \to A$, and we obtain $R \cong A_\varphi$ as $A$-$A$-superbimodule, where $\varphi$ is this homomorphism.
\end{proof}

\begin{proof}[Proof of \cref{prop:picsur}]
Let $A$ be a super algebra and let $P_1, \dots, P_n$ be a collection of the indecomposable, finitely generated,  projective left $A$-modules, one from each isomorphism class. We will show that the super algebra
\begin{equation*}
 E:= \underline{\End}_{A-k}(P)^{\opp}, \qquad P = \bigoplus_{i=1}^n P_i,
\end{equation*}
is Morita equivalent to $A$ and satisfies the assumptions of \cref{Lemma:PicardSurjective}, hence is Picard-surjective.
\\
We claim that $P$ is a generator of the module category $\Bimod_{A-k}$, meaning that for all non-zero homomorphisms $\varphi : M \to N$ of left $A$-modules, there exists $\alpha: P \to M$ such that $\varphi \circ \alpha \neq 0$. 
Assume first that $M = A$, and let $\varphi: A \to N$ be a non-zero homomorphism of left $A$-modules.
Since $A$ is finite-dimensional, it is a direct sum of finitely many indecomposable finitely generated modules. 
Because every direct summand of a projective module is projective each of these summands is moreover projective.
In total, each of these direct summands of $A$ must be isomorphic to one of the $P_i$; 
we obtain
\begin{equation*}
A = P_1^{k_1} \oplus \dots \oplus P_n^{k_n}
\end{equation*}
for some numbers $k_1, \dots, k_n \in \N_0$. 
For $1 \leq i \leq n$, $1 \leq a \leq k_i$, Let $\alpha_{i, a}: P \to M$ be the composition of the projection $P \to P_i$ with the inclusion as the $a$-th $P_i$ summand of $M$. 
If $\varphi \circ \alpha_{i, a} \neq 0$ for some $i$ and $a$, we have detected non-triviality of $\varphi$.
Otherwise, we obtain that $\mathrm{im}\,\alpha_{i, a} \subset \ker(\varphi)$ for all $i$ and $a$. 
However, this implies that $\ker(\varphi) = A$, so $\varphi =0$, a contradiction.
The same argument works in the case that $M = \Pi A$.
Finally, let $M$, $\varphi: M \to N$ a homomorphism be arbitrary and $x \in M$ with $\varphi(x) \neq 0$. 
Splitting $x = x_0 + x_1$ in its homogeneous components, since $\varphi$ is even, we have either $\varphi(x_0) \neq 0$ or $\varphi(x_1) \neq 0$. 
Let $\psi_0: A \to M$ and $\psi_1: \Pi A \to M$ be the homomorphism sending $1$ to $x_0$, respectively $x_1$. 
Now at least one of the compositions $\varphi \circ \psi_0$ or $\varphi \circ \psi_1$ is non-zero, so we have reduced to one of  the previous cases.
\\
Since $P$ is a finitely generated projective generator of the category of right $A$-modules, the functor
\begin{equation*}
F : \Bimod_{A, k} \to \Bimod_{E, k}, \qquad M \to \underline{\Hom}_{A-k}(P, M) 
\end{equation*}
is an equivalence of left module categories \cite[Thm.~3.2.1]{kwok2013super}, hence $E$ is Morita-equivalent to $A$.
$F$ sends the left $A$-module $P$ to the left $E$-module $E$. Since $P$ is a direct sum of indecomposable finitely generated projective modules, one from each isomorphism class, and $F$ is an equivalence, the same is true for $E$. Hence as a left $E$-module,
\begin{equation*}
  E \cong Q_1 \oplus \dots \oplus Q_n
\end{equation*}
for indecomposable, finitely generated, projective left $E$-modules $Q_1, \dots, Q_n$, one from each isomorphism class. This shows that $E$ satisfies the assumptions of \cref{Lemma:PicardSurjective} and finishes the proof.
\end{proof}

\label{sec:crossedmodules}
\label{sec:butterflies}

\end{appendix}

\bibliography{bibfile}
\bibliographystyle{kobib}

\end{document}